\documentclass [12pt,oneside]{amsart}
\usepackage{color}
\usepackage{amssymb,amsmath,amsthm,amsfonts}
\usepackage{enumerate}
\usepackage{setspace,bbm}
\usepackage{empheq}
\usepackage[all]{xy}
\usepackage{verbatim}
\usepackage[normalem]{ulem}
\usepackage[bookmarks,colorlinks,breaklinks]{hyperref}  
\hypersetup{linkcolor=blue,citecolor=blue,filecolor=blue,urlcolor=blue} 

\numberwithin{equation}{section}
\numberwithin{figure}{section}

\newtheorem{theorem}{Theorem}[section]
\newtheorem{proposition}[theorem]{Proposition}

\newtheorem{lemma}[theorem]{Lemma}

\newtheorem{cor}[theorem]{Corollary}

\newtheorem{conjecture}[theorem]{Conjecture}

\theoremstyle{definition}

\newtheorem{remark}[theorem]{Remark}

\definecolor{myblue}{rgb}{0.6, 0.9, 1}

\newcommand{\bigslant}[2]{{\raisebox{.2em}{$#1$}\left/\raisebox{-.2em}{$#2$}\right.}} 

\makeatletter

\newcommand{\Rmnum}[1]{\expandafter\@slowromancap\romannumeral #1@}
\makeatother

\definecolor{myblue}{rgb}{0.6, 0.9, 1}
\definecolor{mygreen}{rgb}{0,0,1}
\definecolor{purple}{rgb}{0.6,0.2,1}
\definecolor{orange}{rgb}{0.8,0,0.2}

\newcommand{\bC}{\mathbb{C}}

\newcommand{\bZ}{\mathbb{Z}}
\newcommand{\bQ}{\mathbb{Q}}

\newcommand{\bN}{\mathbb{N}}

\newcommand{\bA}{\mathbb{A}}
\newcommand{\mcF}{\mathcal{F}}

\newcommand{\mcG}{\mathcal{G}}
\newcommand{\mcL}{\mathcal{L}}

\newcommand{\ord}{\operatorname{ord}}

\newcommand{\pr}{\mathbb{P}}
\newcommand{\al}{\alpha}

\newcommand{\Gal}{\operatorname{Gal}}
\newcommand{\Kbar}{\overline{K}}
\newcommand{\Qbar}{\overline{\bQ}}

\newcommand{\supp}{\operatorname{supp}}
\newcommand{\cL}{\mathcal{L}}
\newcommand{\cS}{\mathcal{S}}

\begin{document}
\title{Variation of canonical height and equidistribution}

\author{Laura De Marco and Niki Myrto Mavraki}

\date{\today}

\begin{abstract}
Let $\pi : E\to B$ be an elliptic surface defined over a number field $K$, where $B$ is a smooth projective curve, and let $P: B \to E$ be a section defined over $K$ with canonical height $\hat{h}_E(P)\not=0$.  In this article, we show that the function $t \mapsto \hat{h}_{E_t}(P_t)$ on $B(\Kbar)$ is the height induced from an adelically metrized line bundle with non-negative curvature on $B$.  Applying theorems of Thuillier and Yuan, we obtain the equidistribution of points $t \in B(\Kbar)$ where $P_t$ is torsion, and we give an explicit description of the limiting distribution on $B(\bC)$.  Finally, combined with results of Masser and Zannier, we show there is a positive lower bound on the height $\hat{h}_{A_t}(P_t)$, after excluding finitely many points $t \in B$, for any ``non-special" section $P$ of a family of abelian varieties $A \to B$ that split as a product of elliptic curves.   
\end{abstract}

\maketitle

\bigskip
\section{Introduction}

Suppose $E \to B$ is an elliptic surface defined over a number field $K$, so $B$ is a smooth projective curve and all but finitely many fibers $E_t$, $t\in B(\overline{K})$, are smooth elliptic curves.  We let $\hat{h}_E$ denote the N\'eron-Tate canonical height of $E$ viewed as an elliptic curve over the function field $k = K(B)$; we let $\hat{h}_{E_t}$ denote the canonical height on the fibers for (all but finitely many) $t\in B(\Kbar)$.  

Suppose that $P \to E$ is a section defined over $K$ for which $\hat{h}_E(P) \not= 0$, so, in particular, the points $P_t$ on the fiber are not torsion in $E_t$ for all $t$.  Tate showed that the function 
	$$t \mapsto \hat{h}_{E_t}(P_t)$$
is a Weil height on $B(\Kbar)$, up to a bounded error  \cite{Tate:variation}.  More precisely, there exists a divisor $D_P \in \mathrm{Pic}(B) \otimes \bQ$ of degree equal to $\hat{h}_E(P)$ so that 
\begin{equation}\label{elliptic variation}
	\hat{h}_{E_t}(P_t) = h_{D_P}(t) + O(1), 
\end{equation}
where $h_{D_P}$ is a Weil height on $B(\Kbar)$ associated to $D_P$.  In a series of three articles \cite{Silverman:VCHI, Silverman:VCHII, Silverman:VCHIII}, Silverman refined statement (\ref{elliptic variation}) by analyzing the N\'eron decomposition of the canonical height on the fibers 
	$$\hat{h}_{E_t}(P_t) = \sum_{v \in M_K} n_v \hat{\lambda}_{E_t, v}(P_t)$$
where $M_K$ denotes the set of places of the number field $K$, and $n_v$ are the integers appearing in the product formula $\prod_{v\in M_K} |x|_v^{n_v} = 1$  for all $x \in K^*$.  

In this article, we explain how Silverman's conclusions about the local functions $\hat{\lambda}_{E_t, v}(P_t)$ are precisely the input needed to show that $t\mapsto \hat{h}_{E_t}(P_t)$ is a ``good" height function on the base curve $B$, from the point of view of equidistribution.  Combining his work with methods from complex dynamics, as in \cite{DWY:Lattes}, and the inequalities of Zhang on successive minima \cite{Zhang:adelic, Zhang:positive}, we prove:  

\begin{theorem}   \label{good height}
Let $K$ be a number field and $k= K(B)$ for a smooth projective curve $B$ defined over $K$.  Fix any elliptic surface $E \to B$ defined over $K$ and point $P \in E(k)$ satisfying $\hat{h}_E(P) \not=0$.  Then 
	$$h_P(t) :=  \hat{h}_{E_t}(P_t),$$
for $t$ with smooth fibers, is the restriction of a height function on $B(\Kbar)$ induced from an adelically metrized ample line bundle $\cL$, with continuous metrics of non-negative curvature, satisfying 
	$$h_P(B) := c_1(\overline{\mathcal{L}})^2/(2 c_1(\mathcal{L})) = 0.$$
\end{theorem}  

Theorem \ref{good height} implies that our height function on $B$ satisfies the hypotheses of the equidistribution theorems of Thuillier and Yuan for points of small height on curves \cite{ChambertLoir:equidistribution, Thuillier:these, Yuan:equidistribution}, and we deduce the following:

\begin{cor} \label{equidistribution on B}  
Let $K$ be a number field and $k= K(B)$ for a smooth projective curve $B$ defined over $K$.  Fix any elliptic surface $E \to B$ defined over $K$ and point $P \in E(k)$ satisfying $\hat{h}_E(P) \not=0$.  There is a collection of probability measures $\mu_P = \{\mu_{P,v}: v \in M_K\}$ on the Berkovich analytifications $B^{an}_v$ such that for any infinite, non-repeating sequence of $t_n \in B(\Kbar)$ such that 
	$$\hat{h}_{E_{t_n}}(P_{t_n}) \to 0$$
as $n\to \infty$, the discrete measures 
	$$\frac{1}{|\Gal(\overline{K}/K) \cdot t_n|}  \sum_{t \in \Gal(\overline{K}/K) \cdot t_n}  \delta_t $$
converge weakly on $B^{an}_v$ to the measure $\mu_{P,v}$ at each place $v$ of $K$.   
\end{cor}

\begin{remark} \label{measures}
The measures $\mu_{P,v}$ of Corollary \ref{equidistribution on B} are not difficult to describe, at least at the archimedean places.  At each archimedean place $v$, there is a canonical positive $(1,1)$-current $T_v$ on the surface $E(\bC)$ (with continuous potentials away from the singular fibers) which restricts to the Haar measure on each smooth fiber $E_t(\bC)$.  The measure $\mu_{P,v}$ on $B(\bC)$ is just the pull-back of this current by the section $P$.  Moreover, at every place, the measure $\mu_{P,v}$ is the Laplacian of the local height function $\hat{\lambda}_{E_t,v}(P_t)$, away from its singularities.  We give more details about (and a dynamical perpective on) the construction of the current $T_v$ in Section \ref{dynamics}.
\end{remark}

As a consequence of Theorem \ref{good height}, and combined with the work of Masser and Zannier \cite{Masser:Zannier, Masser:Zannier:2, Masser:Zannier:nonsimple}, we obtain the so-called Bogomolov extension of their theorems.  Fix integer $m\geq 2$, and suppose that $E_i \to B$ is an elliptic surface over a curve $B$, defined over $\Qbar$, for $i = 1, \ldots, m$.  We consider sections $P$ of the fiber product $A = E_1 \times_B \cdots \times_B E_m$ defined over $\Qbar$.  We say that a section $P = (P_1, P_2, \ldots, P_m)$ is {\em special} if 
\begin{itemize}
\item	for each $i = 1, \ldots, m$, either $P_i$ is torsion on $E_i$ or $\hat{h}_{E_i}(P_i)\not=0$; and
\item for any pair $i,j \in \{1, \ldots, m\}$ such that neither $P_i$ nor $P_j$ is torsion, there are an isogeny $\phi: E_i\to E_j$ and nonzero group endomorphisms $a, b$ of $E_j$ so that $a\circ\phi(P_i) = b(P_j)$. 
\end{itemize}
If a family of abelian surfaces $A \to B$ is isogenous to a fiber product (after performing a base change $B' \to B$ if needed), we say that a section of $A$ is special if it is special on the fiber product.  

It is well known that a special section will always pass through infinitely many torsion points in the fibers $A_t = E_{1,t} \times \cdots \times E_{m,t}$.  That is, there are infinitely many $t \in B(\Qbar)$ for which 
	$$\hat{h}_{E_{1,t}}(P_1(t)) = \cdots = \hat{h}_{E_{m,t}}(P_2(t)) = 0.$$

For a proof see \cite[Chapter 3]{Zannier:book} or, for dynamical proofs, see \cite{D:stableheight}.  

The converse statement is also true, but it is much more difficult:  Masser and Zannier proved that if $\hat{h}_{E_{1,t}}(P_1(t)) = \cdots = \hat{h}_{E_{m,t}}(P_2(t)) = 0$ for infinitely many $t\in B(\Qbar)$, then the section $P$ must be special  \cite{Masser:Zannier:2, Masser:Zannier:nonsimple}.  We extend these results of Masser-Zannier from points of height 0 to points of small height:

\begin{theorem} \label{Zhang language}
Let $B$ be a quasiprojective smooth algebraic curve defined over $\Qbar$.  Suppose $A \to  B$ is a family of abelian varieties of relative dimension $m \geq 2$ defined over $\Qbar$ which is isogeneous to a fibered product of $m\ge 2$ elliptic surfaces.  Let $\cL$ be a line bundle on $A$ which restricts to an ample and symmetric line bundle on each fiber $A_t$, and let $\hat{h}_t$ be the induced N\'eron-Tate canonical height on $A_t$, for each $t\in B(\Qbar)$.   For each non-special section $P: B \to A$ defined over $\Qbar$, there is a constant $c = c(\cL, P) > 0$ so that 
	$$\{t \in B(\Qbar): \hat{h}_t(P_t) < c\}$$ 
is finite.  
\end{theorem}

If $A\to B$ is isotrivial, then Theorem \ref{Zhang language} is a special case of the Bogomolov Conjecture, proved by Ullmo and Zhang \cite{Zhang:Bogomolov, Ullmo:Bogomolov}.  
  
A key ingredient in their proofs is the equidistribution theorem of Szpiro, Ullmo, and Zhang \cite{Szpiro:Ullmo:Zhang}. In his 1998 ICM lecture notes \cite{Zhang:ICM}, Zhang presented a conjecture about geometrically simple families of abelian varieties, which stated, in its most basic form:

\begin{conjecture} [Zhang] \label{Zhang conj} 
Let $B$ be a quasiprojective smooth algebraic curve defined over $\Qbar$.  Suppose $A \to  B$ is a non-isotrivial family of abelian varieties with fiber dimension $> 1$, defined over $\Qbar$ with a simple generic fiber.  Let $\cL$ be a line bundle on $A$ which restricts to an ample and symmetric line bundle on each fiber $A_t$, and let $\hat{h}_t$ be the induced N\'eron-Tate canonical height on $A_t$, for each $t\in B(\Qbar)$.  For each non-torsion section $P: B \to A$ defined over $\Qbar$, there is a constant $c = c(\cL, P) > 0$ so that 
	$$\{t \in B(\Qbar): \hat{h}_t(P_t) < c\}$$ 
is finite.  
\end{conjecture}

When the dimension of the fibers $A_t$ is equal to 2, the finiteness of $\{t \in B(\Qbar): \hat{h}_t(P_t) = 0\}$  for sections as in Conjecture \ref{Zhang conj} was established recently by Masser and Zannier in \cite{Masser:Zannier:simpleA}.  It is well known that the conclusion of Conjecture \ref{Zhang conj} can fail to hold if $A$ is not simple and certainly fails if it is a family of elliptic curves, as mentioned above.  However, the results of Masser and Zannier in their earlier work \cite{Masser:Zannier:2, Masser:Zannier:nonsimple} suggested a formulation of Zhang's conjecture for the non-simple case when $A$ splits as a product of elliptic curves; this is what we proved in our Theorem \ref{Zhang language}.

\begin{remark}
Theorem \ref{good height}, Corollary \ref{equidistribution on B}, and Theorem \ref{Zhang language} were obtained in the special case of the Legendre family $E_t = \{y^2 = x(x-1)(x-t)\}$ over $B= \pr^1$ and the abelian variety $A_t = E_t\times E_t$, for sections $P$ with $x$-coordinates in $\Qbar(t)$ in \cite{DWY:Lattes}, using methods from complex dynamical systems, without appealing to Silverman and Tate's results on the height function.  Moreover, restricting further to sections $P$ with constant $x$-coordinate (in $\pr^1(\overline{\mathbb{Q}})$), Theorem \ref{Zhang language} was obtained without relying on the theorems of Masser and Zannier and gave an alternate proof of their result.  This includes the special case treated by Masser and Zannier in their article \cite{Masser:Zannier}.  For sections with constant $x$-coordinate, the hypothesis on $P$ (that $\hat{h}_E(P) \not=0$) is equivalent to asking that $x(P) \not= 0,1,\infty$ \cite[Proposition 1.4]{DWY:Lattes}.
\end{remark}

%

\bigskip\noindent{\bf Comments and acknowledgements.}  
This project was motivated, in part, by experiments to visualize Silverman's results on the variation of canonical height \cite{Silverman:VCHI, Silverman:VCHII, Silverman:VCHIII} in terms of the measures $\mu_{P,v}$ at archimedean places, and to examine their dependence on $P$.  In particular, the measure detects the failure of the local height function $\hat{\lambda}_{E_t,v}(P_t)$ to be harmonic; compare the comments on non-analyticity preceding Theorem I.0.3 of \cite{Silverman:VCHI}.  The images appearing in  Section \ref{experiment} were first presented at the conference in honor of Silverman's birthday, August 2015.

We thank Charles Favre, Dragos Ghioca, Robert Rumely, Joseph Silverman, and Amaury Thuillier for helpful suggestions.  Our research was supported by the National Science Foundation and the Simons Foundation.

\bigskip
\section{Silverman's work}

\subsection{Preliminaries}  \label{preliminaries}
Let $\mathcal{F}$ be a product formula field of characteristic 0, so there exists a family $M_{\mathcal{F}}$ of non-trivial absolute values on $\mathcal{F}$ and a collection of positive integers $n_v$ for $v \in M_{\mathcal{F}}$ so that 
	$$\prod_{v \in M_{\mathcal{F}}} |x|_v^{n_v} = 1$$
for all $x\in \mathcal{F}^*$.  Let $E/\mathcal{F}$ be an elliptic curve with origin $O$, expressed in Weierstrass form as 
	$$E = \{ y^2+a_1xy+a_3y=x^3+a_2x^2+a_4x+a_6\}$$ 
with discriminant $\Delta$.  Denote by 
	$$\hat{h}_{E}:E(\overline{\mathcal{F}})\to[0,\infty)$$
the N\'eron-Tate canonical height function; it can be defined by 
	$$\hat{h}_E(P) = \frac12 \lim_{n\to\infty} \frac{h(x([n]P))}{n^2}$$
where $h$ is the naive Weil height on $\pr^1$ and $x: E \to \pr^1$ is the degree 2 projection to the $x$-coordinate.  

For each $v\in M_{\mathcal{F}}$, we let $\bC_v$ denote a minimal, algebraically closed field containing $\mathcal{F}$ which is complete with respect to $|\cdot|_v$.  For each $v$, we fix an embedding of $\overline{\mathcal{F}}$ into $\bC_v$.  The canonical height has a decomposition into local heights, as 
	$$\hat{h}_E(P)= \frac{1}{|\Gal( \overline{\mathcal{F}}/\mathcal{F})\cdot P|}  \;
	\sum_{Q \in \Gal( \overline{\mathcal{F}}/\mathcal{F})\cdot P}    \; 
	\sum_{v\in M_{\mathcal{F}}} n_v \, \hat{\lambda}_{E,v}(Q)$$ 
for $P\in E(\overline{\mathcal{F}})\setminus\{O\}$, with the local heights $\hat{\lambda}_{E,v}$ characterized by three properties \cite[Chapter 6, Theorem 1.1, page 455]{Silverman:Advanced}:
\begin{enumerate}
\item	 $\hat{\lambda}_{E,v}$ is continuous on $E(\bC_v)\setminus \{O\}$ and bounded on the complement of any $v$-adic neighborhood of $O$;
\item the limit of $\hat{\lambda}_{E,v}(P) - \frac12 \log|x(P)|_v$ exists as $P \to O$ in $E(\bC_v)$; and
\item for all $P = (x,y) \in E(\bC_v)$ with $[2] P \not= O$, 
	$$\hat{\lambda}([2]P) = 4 \hat{\lambda}(P) -  \log|2y+a_1x+a_3|_v + \frac{1}{4} \log|\Delta|_v.$$
\end{enumerate}

\subsection{Variation of canonical height: the set up}  \label{set up}
Now let $K$ be a number field and $E \to B$ an elliptic surface defined over a number field $K$ with zero section $O:  B \to E$.  Let $P:B\to E$ be a non-zero section defined over $K$, and assume that 
	$$\hat{h}_E(P) \not=0$$
when viewing $P$ as a point on the elliptic curve $E$ defined over $k = \Kbar(B)$.  For each $t\in B(\overline{K})$ such that the fiber $E_t$ is non-singular, we have point $P_t\in E_t(\overline{K})$.  We will investigate the function 
	$$t \mapsto \hat{h}_{E_t}(P_t)$$
which is well defined at all but finitely many $t \in B(\overline{K})$.  Furthermore, via the embedding of $\Kbar$ into $\bC_v$ for each place $v\in M_K$, we may view $E\to B$ as defined over $\bC_v$ and consider the N\'eron local heights $\hat{\lambda}_{E_t, v}(P_t)$ on the non-singular fibers $E_t$ as functions of $t \in B(\bC_v)$.

Let $D_E(P)$ be the $\bQ$-divisor on $B$ defined by 
\begin{equation} \label{divisor definition}
	D_E(P) = \sum_{\gamma \in B(\Kbar)} \hat{\lambda}_{E,\ord_\gamma} (P) \cdot (\gamma).
\end{equation}
Here, $\hat{\lambda}_{E,\ord_\gamma} (P)$ is the local canonical height of the point $P$ on the elliptic curve $E$ over $k = \Kbar(B)$ at the place $\mathrm{ord}_{\gamma}$, for each $\gamma \in B(\Kbar)$.  The degree of $D_E(P)$ is equal to $\hat{h}_E(P)$.  It follows from the definitions that $\supp D_E(P)$ is a subset of the finite set
	$$\{t \in B(\Kbar): E_t \mbox{ is singular}\} \cup \{t \in B(\Kbar): P_t = O_t\}.$$
By enlarging $K$, we may assume that the support of $D_E(P)$ is contained in $B(K)$.

\begin{remark} 
That $D_E(P)$ is a $\bQ$-divisor is standard, following from the fact that the numbers $\hat{\lambda}_{E,\ord_\gamma} (P)$ can be viewed as arithmetic intersection numbers on a N\'eron local model.  See \cite[Chapter~III,~Theorem~9.3]{Silverman:Advanced} for a proof that $\hat{h}_E(P)\in \bQ$; see \cite[Section 6 and p.~203]{Call:Silverman} and \cite[Chapter~11~Theorem~5.1]{Lang:Diophantine} for proofs that each local function $\hat{\lambda}_{E,v}$ also takes values in $\bQ$; see \cite[Theorem 1.1]{DG:rationality} for a dynamical proof.
\end{remark}

\subsection{Variation of canonical height: quasi triviality} 
Let $h_{D_E(P)}$ be an analytic Weil height on $B(\Kbar)$ as defined in \cite[\S3 Example 1(a)]{Silverman:VCHIII}.  That is, we let $g$ be the genus of $B$, and for each point $\gamma \in B(K)$, we choose an element $\xi_\gamma$ of $K(B)$ which has a pole of order $2g+1$ at $\gamma$ and no other poles.  For each non-archimedean place $v$ of $K$, set
	$$ \lambda_{D_E(P), v}(t) = \frac{1}{2g+1} \; \sum_{\gamma \in B(K)} \hat{\lambda}_{E,\ord_\gamma} (P) \; \log^+|\xi_\gamma(t)|_v$$
for all $t \in B(\bC_v) \setminus \supp D_E(P)$.  For archimedean places $v$, the local height is defined by 
	$$ \lambda_{D_E(P), v}(t) = \frac{1}{2(2g+1)}\; \sum_{\gamma \in B(K)} \hat{\lambda}_{E,\ord_\gamma} (P) \; \log\left(1 + |\xi_\gamma(t)|_v^2\right).$$
We set
		$$h_{D_E(P)}(t) = \frac{1}{|\Gal(\Kbar/K)\cdot t|} \; \sum_{s \in \Gal(\Kbar/K)\cdot t} \;
		\sum_{v\in M_K} \lambda_{D_E(P), v}(s)$$ 
for all $t \in B(\Kbar)$.  For fixed choices of $\xi_\gamma$, we will call the associated height $h_{D_E(P)}$ our ``reference height" for the divisor $D_E(P)$.  Silverman proved:

\begin{theorem} \label{Silverman triviality} \cite[Theorem \Rmnum{3}.4.1]{Silverman:VCHIII}
For any choice of reference height $h_{D_E(P)}$, there is a finite set $S$ of places so that 
 	$$\hat{\lambda}_{E_t, v}(P_t) = \lambda_{D_E(P), v}(t)$$
for all $t \in B(\Kbar) \setminus \supp D_E(P)$ and all $v \in M_K \setminus S$.  
\end{theorem}

\subsection{Variation of canonical height: continuity}  
Fix a point $t_0\in B(\Kbar)$ and a uniformizer $u \in \Kbar(B)$ for $t_0$, and consider the function
\begin{equation} \label{variation function}
V_{P, t_0, v}(t):=\hat{\lambda}_{E_t,v}(P_t) + \hat{\lambda}_{E,\ord_{t_0}} (P) \log|u(t)|_v,
\end{equation}
which is not {\em a priori} defined at $t_0$.  Theorem \ref{Silverman triviality} implies that 
	$$V_{P,t_0,v} \equiv 0$$
for all but finitely many places $v$ in a $v$-adic neighborhood of each $t_0$.  
Silverman also proved the following:

\begin{theorem}  \label{Silverman continuity} \cite[Theorem \Rmnum{2}.0.1]{Silverman:VCHII}
Fix $t_0 \in B(\Kbar)$ and a uniformizer $u$ at $t_0$.  For all $v \in M_K$, there exists a neighborhood $U \subset B(\bC_v)$ containing $t_0$ so that the function $V_{P, t_0,v}$ of \eqref{variation function} extends to a continuous function on $U$.
\end{theorem}

\bigskip
\section{A dynamical perspective}
\label{dynamics}

Recall that the N\'eron-Tate height $\hat{h}_E$ and its local counterparts $\hat{\lambda}_{E,v}$ can be defined dynamically.  Letting $E$ be an elliptic curve defined over a number field $K$, the multiplication-by-2 endomorphism $\phi$ on $E$ descends to a rational function of degree 4 on $\pr^1$, via the standard quotient identifying a point $P$ with its additive inverse:
\begin{equation} \label{projection}
\xymatrix{ E \ar[d]_\pi \ar[r]^\phi & E\ar[d]^\pi \\   \pr^1   \ar[r]^{f_\phi} & \pr^1  }
\end{equation}
An elementary, but key, observation is that a point is torsion on $E$ if and only if its quotient in $\pr^1$ is preperiodic for $f_\phi$.  The height $\hat{h}_E$ on $E(\Kbar)$ satisfies
	$$\hat{h}_E (P) = \lim_{n\to\infty} \frac{1}{4^n} h(f_\phi^n(\pi P))$$
where $h$ is the standard logarithmic Weil height on $\pr^1(\Kbar)$.  Now let $E\to B$ be an elliptic surface defined over a number field $K$, and let $P: B\to E$ be a section, also defined over $K$.  In this section, we use this perspective to give a proof of subharmonicity of the local height functions $t\mapsto \hat{\lambda}_{E_t,v}(P_t)$ and the extensions $V_{P, t_0, v}$ of \eqref{variation function}.  We will present this fact as an immediate consequence of now-standard complex-dynamical convergence arguments, at least when the fiber $E_t$ is smooth and the local height $\hat{\lambda}_{E_t,v}(P_t)$ is finite.  Near singular fibers, we utilize the maximum principle and standard results on removable singularities for subharmonic functions.  The same reasoning applies in both archimedean and non-archimedean settings.    

In \S\ref{subsection:measures} we provide the background to justify the explicit description of the limiting distribution $\mu_{P,v}$ at the archimedean places $v$ of $K$, as mentioned in Remark \ref{measures}.

\subsection{Canonical height and escape rates}  \label{ss:dynamics}
As in \S\ref{preliminaries}, we let $E$ be an elliptic curve in Weierstrass form, defined over a product-formula field $\mathcal{F}$ of characteristic 0.  We define a rational function $f = \phi/\psi$ on $\pr^1$ by the formula
	$$f(x(P))=x([2]P)$$ 
for all $P \in E$. Here $x(P)$ is the $x-$coordinate for a point $P\in E$; this function $x$ plays the role of $\pi$ in \eqref{projection}. In coordinates, we have $\phi(x)=x^4-b_4x^2-2b_6x-b_8$ and $\psi(x)=4x^3+b_2x^2+2b_4x+b_6=(2y+a_1x+a_3)^2$ for $P = (x,y)$.  

By a {\em lift} of $f$, we mean any homogeneous polynomial map F on $\bA^2$, defined over $\mcF$, so that $\tau \circ F = f \circ \tau$, where $\tau: \bA^2\setminus \{(0,0)\} \to \pr^1$ is the tautological projection.  A {\em lift} of a point $x \in \pr^1$ is a choice of $X \in \bA^2\setminus \{(0,0)\}$ so that $\tau(X) = x$.

The {\em standard lift} of $f$ will be the map $F:\mathbb{A}^2\to \mathbb{A}^2$ defined by 
\begin{equation}  \label{standard lift}
	F(z,w) = \left(w^4\phi(z/w), w^4\psi(z/w)\right).
\end{equation}
For each $v \in \mathcal{M}_\mcF$, the $v$-adic {\em escape rate} is defined by 
$$\mcG_{F,v}(z,w)=\displaystyle\lim_{\substack{n\to\infty}}\frac{\log||F^n(z,w)||_v}{4^n}$$
where 
	$$\|(z,w)\|_v = \max\{|z|_v, |w|_v\}.$$  
Any other lift of $f$ is of the form $cF$ for some $c \in \mcF^*$; observe that 
$$\mcG_{cF,v} = \mcG_{F,v} + \frac{1}{3} \log|c|_v.$$
Note that 
	$$\mcG_{F,v}(\alpha x, \alpha y) = \mcG_{F,v}(x,y) + \log |\alpha|_v$$
for any choice of lift $F$. Furthermore, $\mcG_{F,v}$ is continuous on $(\bC_v)^2\setminus \{(0,0)\}$, as proved in the archimedean case by \cite{Hubbard:Papadopol, Fornaess:Sibony}.  For non-archimedean absolute values $v$, $\mcG_{F,v}$ extends continuously to the product of Berkovich affine lines $\bA^{1, an}_v \times \bA^{1,an}_v \setminus \{(0,0)\}$ \cite[Chapter 10]{BRbook}.

\begin{proposition}  \label{lambda and G}
For the standard lift $F$ of $f$, and for each place $v$ of $\mcF$, the local canonical height function satisfies
	$$\hat{\lambda}_{E,v}(P) = \frac12 \mcG_{F,v}(x,y) - \frac12 \log|y|_v - \frac{1}{12} \log|\Delta|_v$$
where $x(P) = (x:y)$.
\end{proposition}

\proof
The proof is immediate from the properties of $\mcG_{F,v}$ by checking the three characterizing conditions for $\hat{\lambda}_{E,v}$.  
\qed

\subsection{Variation of canonical height: subharmonicity}
Now let $K$ be a number field and $E \to B$ an elliptic surface defined over a number field $K$ with zero section $O:  B \to E$.  Let $k = \Kbar(B)$; viewing $E$ as an elliptic curve defined over $k$, we also fix a point $P \in E(k)$.    Recall the function $V_{P, t_0, v}(t)$ defined in \eqref{variation function}.  

\begin{theorem}  \label{subharmonic}
For every $t_0 \in B(\Kbar)$ and uniformizer $u$ in $k$ at $t_0$, the function 
	$$V_{P, t_0, v}(t):=\hat{\lambda}_{E_t,v}(P_t) + \hat{\lambda}_{E,\ord_{t_0}} (P) \log|u(t)|_v,$$
extends to a continuous and subharmonic function on a neighborhood of $t_0$ in the Berkovich analytification $B_v^{an}$.
\end{theorem}

The continuity was already established in Theorem \ref{Silverman continuity}, though it was not explicitly stated for the Berkovich space.  The argument below takes care of that.  We begin with a lemma. 

\begin{lemma} \label{harmonic}
Fix $\alpha \in k^*$ and $t_0 \in B(\Kbar)$.  Let $u \in k$ be a uniformizer at $t_0$.  For each place $v$ of $K$, the function 
	$$t\mapsto \log|\alpha_t|_v - (\ord_{t_0} \alpha) \, \log|u(t)|_v$$
is harmonic in a neighborhood of $t_0$ in the Berkovich analytification $B_v^{an}$.
\end{lemma}

\proof
This is Silverman's \cite[Lemma \Rmnum{2}.1.1(c)]{Silverman:VCHII} plus a removable singularities lemma for harmonic functions.  See also \cite[Proposition 7.19]{BRbook} for the extension of a harmonic function to a disk in the Berkovich space $B_v^{an}$. 
\qed

\medskip
Fix $P \in E(k)$.  Let $F$ and $X$ be lifts of $f$ and $x(P)$ to $k^2$, respectively.  Iterating $F$, we set
	$$(A_n, B_n) := F^n(X) \in k^2$$
and observe that
\begin{equation} \label{order growth}
	\mathcal{G}_{F, \ord_{t_0}}(X) = - \lim_{n\to\infty} \frac{\min\{\ord_{t_0} A_n, \ord_{t_0} B_n\}}{4^n}
\end{equation}
from the definition of the escape rate.  We let $F_t$ and $X_t$ denote the specializations of $F$ and $X$ at a point $t \in B(\Kbar)$; they are well defined for all but finitely many $t$.  Observe that if $F$ is the standard lift for $E$ then so is $F_t$ for all $t$.

\begin{proposition}  \label{escape rate subharmonic}
Fix $P \in E(k)$, $t_0 \in B(\Kbar)$, and $v\in M_K$.  For any choice of lifts $F$ of $f$ and $X$ of $x(P)$, the function 
	$$G_P(t;v) := \mathcal{G}_{F_t, v}(X_t) + \mathcal{G}_{F, t_0}(X) \log|u(t)|_v$$
extends to a continuous and subharmonic function in a neighborhood of $t_0$ in $B_v^{an}$.  
\end{proposition}

\proof
First observe that the conclusion does not depend on the choices of $F$ and $X$.  Indeed, 
\begin{eqnarray*}
\mathcal{G}_{c_t F_t, v}(\alpha_t X_t) + \mathcal{G}_{c F, t_0}(\alpha X) \log|u(t)|_v &=& 
\mathcal{G}_{F_t, v}(X_t) + \mathcal{G}_{F, t_0}(X) \log|u(t)|_v  \\
&&	+ \; \frac{1}{3}\left( \log|c_t|_v - (\ord_{t_0}c) \log|u(t)|_v\right) \\
&&	+ \; \log|\alpha_t|_v - (\ord_{t_0}\alpha)\log|u(t)|_v 
\end{eqnarray*}
for any $c, \alpha \in k^*$.  So by Lemma \ref{harmonic} the function $G_P(t; v)$ is continuous and subharmonic for one choice if and only if it is continuous and subharmonic for all choices.   

Let $F$ be the standard lift of $f$.  Suppose that $P=O$.  Since $F(1,0) = (1,0)$, we compute that 
 $$G_{O}(t;v)=  \mathcal{G}_{{F_t}, v}(1,0)+ \mathcal{G}_{F, t_0}(1,0) \log|u(t)|_v \equiv 0.$$

Now suppose that $P \neq O$.  Fix $t_0\in B(\overline{K})$ and local uniformizer $u$ at $t_0$.  Choose a lift $F$ of $f$ so that the coefficients of $F$ have no poles at $t_0$, with $F_{t_0} \not= (0,0)$.  Choose lift $X$ of $x(P)$ so that $X_t$ is well defined for all $t$ near $t_0$ and $X_{t_0} \not= (0,0)$.  As above, we write 
	$$F^n(X) = (A_n, B_n)$$
and put 
	$$a_n = \min\{\ord_{t_0} A_n, \ord_{t_0} B_n\}$$
so that $a_n\geq 0$ for all $n$ and $a_0 = 0$.  Set
      $$F_n(t) = F_t^n(X_t)/u(t)^{a_n}.$$
For each place $v$, we set
	$$h_{n,v}(t) = \frac{\log\|F_n(t)\|_v}{4^n}.$$
By construction, the limit of $h_{n,v}$ (for $t$ near $t_0$ with $t\not=t_0$) is exactly the function $G_P$ for these choices.  In fact, for $t$ in a small neighborhood of $t_0$, but with $t\not= t_0$, the function $f_t$ on $\pr^1$ is a well-defined rational function of degree 4; so the specialization of the homogeneous polynomial map $F_t$ satisfies $F_t^{-1}\{(0,0)\} = \{(0,0)\}$.  Furthermore, as the coefficients of $F_t$ depend analytically on $t$, the functions $h_{n,v}$ converge locally uniformly to the function $G_P$ away from $t=t_0$.  This can be seen with a standard telescoping sum argument, used often in complex dynamics, as in \cite[Proposition 1.2]{Branner:Hubbard:1}.  In particular, $G_P$ is continuous on a punctured neighborhood of $t_0$.  

At the archimedean places $v$, and for each $n$, the function $h_{n,v}$ is clearly continuous and subharmonic in a neighborhood of $t_0$.  At non-archimedean places $v$, this definition extends to a Berkovich disk around $t_0$, setting 
	$$h_{n,v}(t) = \frac{1}{4^n} \max\{\log[A_n(T)/T^{a_n}]_t, \log[B_n(T)/T^{a_n}]_t\}$$
where $[\cdot]_t$ is the seminorm on $K[[T]]$ associated to the point $t$.  Each of these functions $h_{n,v}$ is continuous and subharmonic for $t$ in a Berkovich disk around $t_0$.  (Compare \cite{BRbook} Example 8.7, Proposition 8.26(D), and equation (10.9).)  

\begin{lemma} \label{upper bound}
For all $v$, and by shrinking the radius $r$ if necessary, the functions $h_{n,v}$ are uniformly bounded from above on the (Berkovich) disk $D_r$.  
\end{lemma}

\proof
As observed above, the functions $h_{n,v}$ converge locally uniformly away from $t=t_0$ to the continuous function $G_P(t)$.  Choose a small radius $r$, and let 
	$$M_v = \sup_n \max_{|t|_v=r} h_{n,v}(t)$$
which is finite by the convergence.  Because the functions are subharmonic, the Maximum Principle implies that $h_{n,v}(t) \leq M_v$ throughout the disk of radius $r$, for all $n$.  For the non-archimedean places, there is also a Maximum Principle on the Berkovich disk, where the role of the circle of radius $r$ is played by the Type II point associated to the disk of radius $r$ (see \cite{BRbook} Proposition 8.14).  
\qed

\medskip
We can now complete the proof of Proposition \ref{escape rate subharmonic}.  As each $h_{n,v}$ is subharmonic, and the functions are uniformly bounded from above on the disk by Lemma \ref{upper bound}, we know that the (upper-semicontinuous regularlization of the) limsup of these functions is subharmonic.  See \cite{BRbook} Proposition 8.26(F) for a proof in the non-archimedean case. 
\qed

\medskip
\proof[Proof of Theorem \ref{subharmonic}]
Subharmonicity now follows from Proposition \ref{lambda and G}, Lemma \ref{harmonic}, and Proposition \ref{escape rate subharmonic}.  The continuity at each archimedean place is the content of Theorem \ref{Silverman continuity}.  The continuity at each non-archimedean place is a combination of the continuity on the punctured Berkovich disk (as in the proof of Proposition \ref{escape rate subharmonic}) and the continuity on Type I (classical) points given in Theorem \ref{Silverman continuity}.  
\qed

\subsection{The measures on the base}  \label{subsection:measures}
Here we provide more details about the description of the measures appearing in the statement of Corollary \ref{equidistribution on B}, as discussed in Remark \ref{measures}.  

Fix an archimedean place $v$ and any point $t_0 \in B(\Kbar)$.  Choosing a uniformizer $u$ at $t_0$, recall the definition of $V_{P, t_0, v}$ from \eqref{variation function}.  We define 
	$$\mu_{P,v} := dd^c V_{P,t_0, v}(t)$$
on a neighborhood of $t_0$ in $B_v^{an}$; note that this is indepedent of the choice of $u$.  Note that $\mu_{P,v}$ can be expressed as 
	$$\mu_{P,v} = dd^c \hat{\lambda}_{E_t, v}(P_t)$$
for $t$ outside of the finitely many points in the support of the divisor $D_E(P)$ or where the fiber $E_t$ is singular.  Note, further, that $\mu_{P,v}$ assigns no mass to any individual point $t_0$, because the potentials are bounded by Theorem \ref{subharmonic}.  The details on the metric and the equidistribution theorem in Section \ref{sec:metric} will show that these are exactly the measures that arise as the distribution of the points of small height in Corollary \ref{equidistribution on B}.

It is well known that the local height function on a smooth elliptic curve is a potential for the Haar measure.  That is, for fixed $t$ we have
	$$dd^c \hat{\lambda}_{E_t, v}(\cdot) = \omega_t - \delta_o$$
where $\omega_t$ is the normalized Haar measure on $E_t$ and $\delta_o$ is a delta-mass supported at the origin of $E_t$; see, e.g., \cite[Theorem II.5.1]{Lang:Arakelov}.  We present an alternative proof of this fact related to dynamics as part of Proposition \ref{measure description}, as a consequence of Proposition \ref{lambda and G}.  

\begin{proposition}  \label{measure description}
Let $E \to B$ be an elliptic surface and $P: B \to E$ a section, both defined over a number field $K$.  Let $\cS \subset E$ be the union of the finitely many singular fibers in $E$.  For each archimedean place $v$ of $K$, there is a positive, closed (1,1) current $T_v$ on $E\setminus \cS$ with locally continuous potentials so that $T_v|_{E_t}$ is the Haar measure on each smooth fiber, and $P^* T_v$ is equal to the measure $\mu_{P,v}$.   
\end{proposition}

\begin{remark}
As $T_v$ has continuous potentials, the restriction $T_v|_{E_t}$ and the pullback $P^* T_v$ are well defined.  That is, we have $T_v|_{E_t} = dd^c (u|_{E_t})$ where $u$ is a locally defined potential of $T_v$, and $P^*T_v = dd^c (u\circ P)$ locally on $B$.  The measure $\mu_{P,v}$ has no atoms, so it is determined by $T_v$ along the image of $P$ in $E\setminus \cS$.  
\end{remark}

\proof[Proof of Proposition \ref{measure description}]
Let us fix any small neighborhood $U$ in the base curve $B(\bC)$ so that all fibers $E_t$ are smooth for $t\in U$.  Let $f_t$ be the map on $\pr^1$ defined in \S\ref{ss:dynamics}; by shrinking $U$ if necessary, we can find lifts $F_t$ of $f_t$ that are holomorphic in $t \in U$.  From \cite{Hubbard:Papadopol, Fornaess:Sibony} (or the proof of \cite[Proposition 1.2]{Branner:Hubbard:1}), we know that the escape rate
	$$\mcG_{F_t,v}(z,w)=\displaystyle\lim_{n\to\infty}\frac{\log||F_t^n(z,w)||_v}{4^n}$$
is continuous and plurisubharmonic as a function of $(t, z,w) \in U\times (\bC^2\setminus \{(0,0)\})$.  Therefore
	$$dd^c \mcG_{F_t,v}(z,w)$$
projects to a closed and positive (1,1)-current $G_v$ on the complex surface $U\times\pr^1$, with locally continuous potentials.  This current $G_v$ has the property that, restricted to each fiber $\pr^1$, its total mass is 1; and the measure on the fiber is the measure of maximal entropy for the rational map $f_t$ \cite{Lyubich:entropy, Hubbard:Papadopol}.  

The restriction $E|_U$ of the elliptic surface $E$ over $U$ maps with degree 2 to the complex surface $U\times \pr^1$ by the projection $\pi$ of \eqref{projection}.  The current $G_v$ can be pulled back to $E$ as $\frac12 \, dd^c(g \circ \pi)$ where $g$ is a locally-defined continuous potential for $G_v$.  Covering the base of $E\setminus \cS$ by sets of the form $U$, the local definitions glue to form the closed, positive $(1,1)$-current $T_v$ on $E\setminus \cS$.  

If $P: B \to E$ is a section defined over the number field $K$, then $P^* T_v$ has potential given locally by 
	$$\frac12\;  g\circ \pi \circ P =  \frac12 \; \mcG_{F_t, v}(X_t)$$
for any lift $X_t$ of $\pi(P_t) \in \pr^1$.  Proposition \ref{lambda and G} yields that $P^* T_v$ must coincide with $\mu_{P,v}$.  

Finally, to conclude that $T_v|_{E_t}$ is equal to the normalized Haar measure $\omega_t$, we may use the well-known dynamical fact that for each fixed $t$ in the base, the measure $\omega_t$ projects by $\pi$ to $\pr^1$ to the unique measure of maximal entropy for the map $f_t$; see, e.g., \cite[\S 7]{Milnor:Lattes}.
\qed

\bigskip
\section{The adelic metric and equidistribution}
\label{sec:metric}

In this section we give the proofs of Theorem \ref{good height} and Corollary \ref{equidistribution on B}.  

We first outline the proofs.  Let $E \to B$ be an elliptic surface defined over a number field $K$ with zero section $O:  B \to E$, and let $P:B\to E$ be a section also defined over $K$ so that $\hat{h}_E(P) \not=0$.  Recall from \S\ref{set up} that we introduced a $\bQ$-divisor 
	$$D_E(P) = \sum_{\gamma \in B(\Kbar)} \hat{\lambda}_{E, \ord_\gamma}(P) \cdot (\gamma)$$
on $B$.  By enlarging $K$, we may assume that $\supp D_E(P)$ lies in $B(K)$.  We will define an adelic metric on the ample line bundle $\cL$ associated to the divisor $D_E(P)$, inducing a height function $h_{\cL}$ such that
	$$h_{\cL}(t) = \hat{h}_{E_t}(P_t) \mbox{ for all but finitely many } t \in B(\Kbar)$$
and
	$$h_{\cL}(t) \geq 0 \mbox{ for all } t \in B(\Kbar).$$
Applying Silverman's results on the variation of canonical height, Theorems \ref{Silverman triviality} and \ref{Silverman continuity}, we will deduce that the metric is continuous and adelic.  From Theorem \ref{subharmonic}, we will conclude that the metric is also semi-positive in the sense of Zhang \cite{Zhang:positive}.  We will use Zhang's inequalities \cite{Zhang:adelic} to deduce that 
	$$h_{\cL}(B) = 0.$$
Consequently, we will be able to apply the equidistribution results of Chambert-Loir, Thuillier, and Yuan \cite{ChambertLoir:equidistribution, Thuillier:these, Yuan:equidistribution} to complete our proofs.

\subsection{The metric and its properties} \label{metric definition}
Let $m\in\mathbb{N}$ be such that 
	$$D = m\cdot D_E(P)$$
is an integral divisor.  Let $\mcL_m$ be the associated line bundle on $B$.  Note that $\deg(\mcL_m)=m \, \hat{h}_E(P)>0$ so $\mathcal{L}_m$ is ample; by replacing $m$ with a multiple, we may assume that $\mcL_m$ is very ample.  

Fix a place $v$ of $K$.  Let $U$ be an open subset of $B_v^{an}$.  Each section $s \in \mcL_m(U)$ is identified with a meromorphic function $f$ on $U$ satisfying 
	$$(f) \geq -D.$$
We set
	$$\|s(t)\|_v = \left\{ \begin{array}{ll} 
	e^{-m\hat{\lambda}_{E_t,v}(P_t)}|f(t)|_v & \mbox{if } f(t) \mbox{ is finite and nonzero } \\
	0 & \mbox{if } \ord_t f > - m \, \hat{\lambda}_{E, \ord_t}(P) \\
	e^{-m \,V_{P, t, v}(t)} & \mbox{otherwise}, 
	\end{array} \right.  $$
taking the locally-defined uniformizer $u = f^{1/\ord_t f}$ at $t$ in the definition of $V_{P, t, v}$ from \eqref{variation function}.  

\begin{theorem}  \label{metric properties}
The metric $\|\cdot \| = \{\|\cdot\|_v\}_{v \in M_K}$ on $\mcL_m$ is continuous, semipositive, and adelic.
\end{theorem}

\proof
The continuity and semipositivity follows from Theorem \ref{subharmonic}.  (In \cite{ChambertLoir:survey}, semipositivity of a continuously metrized line bundle on a curve is defined terms of subharmonicity of potentials for the curvature form at each archimedean place, and as a uniform limit of ``smooth semipositive" metrics at each non-archimedean place.  In \cite{Thuillier:these}, it is established that subharmonicity of potentials is a sufficient notion at all places, and he proves in \cite[Theorem 4.3.3]{Thuillier:these} that this notion of semipositivity coincides with that of Zhang \cite{Zhang:positive}.  See also \cite[Lemma 3.11, Theorem 3.12]{Favre:Gauthier:cubics} where this same argument is applied in a dynamical context.)  The adelic condition follows from Theorem \ref{Silverman triviality}.  
\qed

\subsection{The associated height function}
A height function on $B(\Kbar)$ is defined by setting 
	$$h_P(t) :=  \frac{1}{m} \, \frac{1}{| \Gal(\Kbar/K)\cdot t|} \, \sum_{s \in \Gal(\Kbar/K)\cdot t} \sum_{v\in M_K} -n_v \, \log\|\phi(s)\|_v$$
where $\phi$ is any global section of $\mcL_m$ which is nonvanishing along the Galois orbit of $t$, and $\|\cdot\|_v$ is the metric of \S\ref{metric definition}.   Recall that $\supp D_E(P) \subset B(K)$; we may assume that our sections $\phi$ are defined over $K$, and the product formula guarantees our height is independent of the choice of $\phi$.  

Our next goal is to prove the following two important facts about this height function $h_P$.  

\begin{proposition} \label{heightvsheight}
The height function $h_P$ satisfies
	$$h_P(t)= \hat{h}_{E_t}(P_t)$$
for all $t\in B(\Kbar)$ such that the fiber $E_t$ is smooth.
\end{proposition}

\begin{proposition}  \label{height is nonnegative}
The height function $h_P$ satisfies
	$$h_P(t) \geq 0$$
for all $t \in B(\Kbar)$.
\end{proposition}

\proof[Proof of Proposition \ref{heightvsheight}]
First fix $t\in B(\Kbar)\setminus\supp D_E(P)$ with smooth fiber $E_t$.  Choose a section $\phi$ defined over $K$ that does not vanish along the Galois orbit of $t$, and let $f$ be the associated meromorphic function on $B$.  Then $f$ takes finite and nonzero values along the Galois orbit of $t$.  We have, 
\begin{eqnarray*}
	h_P(t) &=& \frac{1}{m}  \; \frac{1}{| \Gal(\Kbar/K)\cdot t|}  \sum_{s \in \Gal(\Kbar/K)\cdot t} \;
			\sum_{v \in M_K} \; n_v \, (m\, \hat{\lambda}_{E_s, v}(P_s) - \log|f(s)|_v) \\
		&=& \frac{1}{m}  \; \frac{1}{| \Gal(\Kbar/K)\cdot t|}  \sum_{s \in \Gal(\Kbar/K)\cdot t} \;
			\sum_{v \in M_K} \; m\, n_v \, \hat{\lambda}_{E_s, v}(P_s) \\
		&=& \hat{h}_{E_t}(P_t).
\end{eqnarray*}
where the second equality follows from the product formula.  

For $t_0\in \supp D_E(P)$ such that $E_{t_0}$ is smooth, it is necessarily the case that $P_{t_0} = O_{t_0}$, and therefore $\hat{h}_{E_{t_0}}(P_{t_0}) = 0$.  To compute $h_P(t_0)$, observe that $t_0 \in B(K)$ so its Galois orbit is trivial; fixing a uniformizer $u\in K(B)$ at $t_0$, we have 
	$$h_P(t_0) = \sum_{v\in M_K}  n_v \, V_{P, t_0, v}(t_0)$$
where $V_{P,t_0, v}$ is the function of \eqref{variation function} associated to the uniformizer $u$. 

We can compute $h_P(t_0)$ using the dynamical interpretation of the local heights, described in Section \ref{ss:dynamics}.  Fix a Weierstrass equation for $E$ in a neighborhood of $t_0$ and write $P=(x_P,y_P)$. The assumption that $P_{t_0}=O_{t_0}$ is equivalent to $\ord_{t_0}x_P<0$.  After possibly shrinking $U$, write $x_P=(u)^{\ord_{t_0}(x_P)}A_0$ for the chosen uniformizer $u$ and a function $A_0\in K(B)$ that does not vanish in $U$. We choose a lift $X$ of $x_P$ on $U$ defined as $X=(A_0,B_0)$, where $B_0:=(u)^{-\ord_{t_0}(x_P)}$. Notice that $A_0$ and $B_0$ are regular at $t_0$.  Let $F$ be the standard lift in these coordinates, defined in \eqref{standard lift}; it satisfies $F_{t_0}(1,0) = (1,0)$, and we have $\mcG_{F,\ord_{t_0}}(A_0,B_0)=0$.  Since $\ord_{t_0}\Delta_E=0$, Proposition \ref{lambda and G} implies that
	$$V_{P,t_0,v}(t)=\frac12 \, \mcG_{F_t,v}(A_0(t),B_0(t))-\frac{1}{12}\, \log|\Delta_E(t)|_v$$
for all $t \in U$.  Therefore, 
\begin{eqnarray*}
V_{P, t_0, v}(t_0) &=& \frac12 \, \mcG_{F_{t_0},v}(A_0(t_0),0) -\frac{1}{12}\, \log|\Delta_E(t_0)|_v \\
	&=& \frac12 \, \lim_{n\to\infty}\frac{1}{4^n}\log \|F^n_{t_0}(A_0(t_0),0)\|_v - \frac{1}{12}\, \log|\Delta_E(t_0)|_v \\
	&=& \frac12 \, \lim_{n\to\infty}\frac{1}{4^n}\log |A_0(t_0)^{4^n}|_v - \frac{1}{12}\, \log|\Delta_E(t_0)|_v \\
	&=& \frac12 \, \log|A_0(t_0)|_v - \frac{1}{12} \, \log|\Delta_E(t_0)|_v.
\end{eqnarray*}
The product formula now yields that $h_P(t_0)=0$, as claimed.
\qed

\medskip
To prove Proposition \ref{height is nonnegative}, we first reduce to the case that the elliptic surface $E\to B$ has semi-stable reduction; that is all of its fibers are either smooth or have multiplicative reduction.  The next lemma describes how the height associated with the divisor $D_E(P)$ behaves under base extensions of the elliptic surface $E\to B$. It is adapted from \cite[Reduction Lemma \Rmnum{2}.2.1]{Silverman:VCHII}. We include it here for completeness.

\begin{lemma}\label{reductionlemma}
Let $\mu: B'\to B$ be a finite map of smooth projective curves, let $E'\to B'$ be a minimal model for $E\times_B B'$, and let $P':B'\to E'$ be the extension of the section $P$. For each $t_0 \in B(\Kbar)$ and $t_0'\in \mu^{-1}(\{t_0\})\subset B'(\bC_v)$, there is a neighborhood $U$ of $t_0'$ in $B'(\bC_v)$ and a regular non-vanishing function $f$ on $U$ such that 
$$V_{P,t_0,v}(\mu(t'))-V_{P',t_0',v}(t')=\log|f(t')|_v$$
on $U\setminus\{t'_0\}$. 
In particular,
$$V_{P,t_0,v}(t_0)-V_{P',t_0',v}(t'_0)=\log|f(t'_0)|_v.$$
\end{lemma}

\begin{proof}
Let $u$ be a uniformizer at $t_0$, $u'$ a uniformizer at $t_0'$ and $n=\ord_{t'_0}(\mu^{*}u)$. Since local heights are invariant under base extension we have 
\begin{align}\label{generic:invariance}
\hat{\lambda}_{E', \ord_{t'_0}}(P')=n\, \hat{\lambda}_{E, \ord_{t_0}}(P).
\end{align}
Notice that for all $t'$ in a punctured neighborhood of $t'_0$ the fibers $E'_{t'}$ are smooth. Hence the map $E'\to E$ gives an isomorphism between the fibers $E'_{t'}\to E_{\mu(t')}$. Under this isomorphism $P'_{t'}\in E'_{t'}$ is mapped to $P_{\mu(t')}\in E_{\mu(t')}$. Invoking now the uniqueness of the N\'eron local heights, we have 
\begin{align}\label{special:isomorphism}
\hat{\lambda}_{E_{\mu(t')},v}(P_{\mu(t')})=\hat{\lambda}_{E'_{t'},v}(P'_{t'}).
\end{align}
Combining \eqref{generic:invariance} and \eqref{special:isomorphism} we get that for $t'$ in a punctured neighborhood of $t_0'$, 
$$V_{P,t_0,v}(\mu(t'))=V_{P',t_0',v}(t')+\hat{\lambda}_{E,\ord_{t_0}}(P)\log\left|\frac{u(\mu(t'))}{u'^n(t')}\right|_v.$$
The definition of $n$ yields that the function $f(t')=\left(\frac{u(\mu(t'))}{u'^n(t')}\right)^{\hat{\lambda}_{E, \ord_{t_0}}(P)}$ is regular and non-vanishing at $t_0'$. The first part of the lemma follows. 

Finally, Theorem \ref{Silverman continuity} allows us to conclude that 
$$V_{P,t_0,v}(\mu(t'_0))-V_{P',t_0',v}(t'_0)=\log|f(t'_0)|_v$$
at the point $t'_0$, as claimed.  
\end{proof}

The following lemma will allow us to prove Proposition \ref{height is nonnegative} in the case that a fiber has multiplicative reduction. The proof is lengthy, but it is merely a collection of computations using the explicit formulas for the local height functions, as in \cite[Theorem \Rmnum{6}.3.4, \Rmnum{6}.4.2]{Silverman:Advanced}.

\begin{lemma} \label{bad fibers}  
Let $E\to B$ be an elliptic surface and let $P:B\to E$ be a non-zero section defined over $K$. Then there exists a finite extension $L$ of the number field $K$ so that, for each $t_0\in B(\Kbar)$ such that $E_{t_0}$ has multiplicative reduction, there exists an $x(t_0) \in L^*$ so that 
	$$V_{P,t_0,v}(t_0) = \log|x(t_0)|_v$$
at all places $v$ of $L$.  
\end{lemma}

\begin{proof}
We let
\begin{align}\label{weierstrasseq}
E:y^2=x^3+ax+b,
\end{align}
be a minimal Weierstrass equation for $E$ over an affine subset $W\subset B$ defined over $K$ with $t_0\in W$. Here $a,b\in K(B)$ are regular functions at $t_0$. 
Using this Weierstrass equation we write 
$$P=(x_P,y_P),$$
where $x_P,y_P\in K(B)$.
Since $E\to B$ has multiplicative reduction over $t_0\in B(K)$, we have 
\begin{align}\label{multiplicativered}
N:=\ord_{t_0}\Delta_E\ge 1\text{ and }\min\{\ord_{t_0}a,\ord_{t_0}b\}=0.
\end{align}
Let $v$ be a place of $K$ (archimedean or non-archimedean). We denote by $j_E$ the $j-$invariant of $E\to W$, given by $$j_E(t)=1728\frac{(4a(t))^3}{\Delta_E(t)}.$$ 
Notice that equation \eqref{multiplicativered} yields that $j_E$ has a pole at $t_0$. Hence, we can find a $v-$adic open neighborhood $U$ of $t_0$ and an analytic map 
\begin{align*}
\psi: U\to\{q\in\bC_v~:~|q|_v<1\},
\end{align*}
such that the following holds: If $j$ is the modular $j-$invariant \cite[Chapter \Rmnum{5}]{Silverman:Advanced}, then we have
\begin{align*}
j_E(t)=j(\psi(t))\text{ and } \ord_{t_0}\psi=N.
\end{align*}
The function $\psi(t)$ is given as 
\begin{align}\label{psiwithj}
\psi(t)=\frac{1}{j_E(t)}+\frac{744}{j^2_E(t)}+\frac{750420}{j^3_E(t)}+\ldots \in \mathbb{Z}[[(j_E(t))^{-1} ]].
\end{align}
In the following, we choose a uniformizer $u \in K(B)$ at $t_0$, and we identify $\psi$ with its expression $\psi(t)\in \bC_v[[u]]$ and write 
\begin{align}\label{psi(t)}
\psi(t)=\beta u(t)^{N}+u(t)^{N+1}f(t)\text{, for }t\in U\setminus\{t_0\}.
\end{align}
Equation \eqref{psiwithj} yields that $\beta\in K\setminus\{0\}$ and $f(t)\in K[[u]]$. 
Following the proof of  \cite[Section 6]{Silverman:VCHII} and after possibly shrinking $U$ we have isomorphisms
\begin{align}\label{uni-iso}
 E_t(\bC_v)&\xrightarrow{\sim} \bigslant{\bC_v^{*}}{\psi(t)^{\bZ}}\xrightarrow{\sim} C_{\psi(t)}:y^2=4x^3-g_2(\psi(t))x-g_3(\psi(t))\text{,}
\end{align}
for $t\in U\setminus\{t_0\}$.
Under these isomorphisms, we have
\begin{align*}
 P_t&\mapsto \quad w(t)\quad\mapsto (\wp(w(t),\psi(t)),\wp'(w(t),\psi(t))).
\end{align*}
Here $g_2,g_3$ are the modular invariants, given by their usual $q-$series 
\begin{align*}
g_2(q)=\frac{1}{12}\left(1+240\displaystyle\sum_{n=1}^{\infty}\frac{n^3q^n}{1-q^n}\right)\text{, }g_3(q)=\frac{1}{216}\left(-1+504\displaystyle\sum_{n=1}^{\infty}\frac{n^5q^n}{1-q^n}\right)
\end{align*}
and $\wp$ is the Weierstrass $\wp-$function given by
\begin{align}\label{weierstrassp}
\wp(w,q)=\frac{1}{12}+\sum_{n\in\bZ}\frac{q^nw}{(1-q^nw)^2}-2\sum_{n=1}^{\infty}\frac{nq^n}{1-q^n}\text{, }\wp'(w,q)=\sum_{n\in\bZ}\frac{q^nw(1+q^nw)}{(1-q^nw)^3}.
\end{align}
In view of \cite[Lemma \Rmnum{2}.6.2]{Silverman:VCHII}, after possibly replacing $P$ by $-P$, we may assume that $w:U\to\bC_v$ is an analytic map satisfying
\begin{align}\label{orderofwandq}
0\le\ord_{t_0}w\le\frac{1}{2}\ord_{t_0}\psi.
\end{align}
In the following we identify $w$ with its series in $\bC_v[[u]]$ and write 
\begin{align}\label{w(t)}
w(t)=\al u^m(t)+u^{m+1}(t)g(t),
\end{align}
where $\al\in\bC_v$ and $g(t)\in\bC_v[[u]]$.

We claim that $w(t)\in \overline{K}[[u]]$. To see this, notice that from \cite[\Rmnum{3}]{Silverman:Elliptic} we have that for $t\in U$  
\begin{align*}
(\wp(w(t),\psi(t)),\wp'(w(t),\psi(t))=(\nu^{-2}(t)x_P(t),2\nu^{-3}(t)y_P(t)), 
\end{align*}
where
\begin{align*}
\nu(t)^{12}=\frac{\Delta_E(t)}{\Delta(\psi(t))}.
\end{align*}
In the equation above $\Delta$ denotes the modular discriminant given by
\begin{align*}
\Delta(q)=g_2(q)^3-27g_3(q)^3.
\end{align*}
Since the functions $\psi,\Delta_E$ and $\Delta$ are defined over $K$, we have that $Y(t):=2\nu^{-3}(t)y_P(t)$ is also defined over $K$.  Since $Y(t)=\wp'(w(t),\psi(t))\in K[[u]]$ and $\psi(t)\in K[[u]]$ we get that $w(t)\in K[[u]]$.

Therefore, there are non-zero constants $\al,\beta,\gamma\in\overline{K}\setminus\{0\}$, non-negative integers $k,m\in\bN$ and functions $f(t), g(t), h(t)\in \overline{K}[[u]]$ such that for all $t\in U$ 
\begin{align}\label{abcmNk}
\psi(t)=\beta u^N(t)+f(t)u^{N+1}(t)\text{, }w(t)=\al u^m(t)+g(t)u^{m+1}\text{, }1-w(t)=\gamma u^k(t)+h(t)u^{k+1}(t).
\end{align}

Next, we aim to express $x(t_0)$ (as in the statement of the lemma) in terms of $\al,\beta,\gamma\in\overline{K}$.

Using the isomorphisms in \ref{uni-iso}, the uniqueness of the local canonical heights and the explicit formulas for the local canonical heights \cite[Theorem \Rmnum{6}.3.4, \Rmnum{6}.4.2]{Silverman:Advanced}, we get
\begin{align}\label{bernoulliformula}
\hat{\lambda}_{E_t,v}(P_t)=\hat{\lambda}(w(t),\psi(t))&=-\frac{1}{2}B_2\left(\frac{\log|w(t)|_v}{\log|\psi(t)|_v}\right)\log|\psi(t)|_v-\log|1-w(t)|_v\\
&-\sum_{n\ge 1}\log|(1-\psi(t)^nw(t))(1-\psi(t)^nw(t)^{-1})|_v,
\end{align}
where $B_2(s)=s^2-s+1/6$ is the second Bernoulli polynomial.

Since $\ord_{t_0}\psi=N\ge 1$ and using \eqref{orderofwandq}, we get
\begin{align}\label{bigsumzero}
\displaystyle\lim_{t\stackrel{v}\to t_0}\sum_{n\ge 1}\log|(1-\psi(t)^nw(t))(1-\psi(t)^nw(t)^{-1})|_v=0.
\end{align}
In what follows, for $F(t)\in\bC_v[[u]]$ we write
$$F(t):=o_v(1)\text{, if }\displaystyle\lim_{t\stackrel{v}\to t_0}F(t)=0.$$
In view of \cite[Lemma \Rmnum{1}.5.1]{Silverman:VCHI}, we have
\begin{align}\label{bernoullipart}
B_2\left(\frac{\log|w(t)|_v}{\log|\psi(t)|_v}\right)\log|\psi(t)|_v&=\frac{\log^2|w(t)|_v}{\log|\psi(t)|_v}-\log|w(t)|_v+\frac{1}{6}\log|\psi(t)|_v\\\nonumber
&=\frac{m^2}{N}\log|u(t)|_v+\frac{m}{N^2}\log\left(\frac{|\al|_v^{2N}}{|\beta|_v^{m}}\right)-\log|\al|_v\\\nonumber
&-m\log|u(t)|_v+\frac{\log|\beta|_v}{6}+\frac{N}{6}\log|u(t)|_v+o_v(1)
\end{align}
Using equations  \eqref{bigsumzero} and \eqref{bernoullipart}, equation \eqref{bernoulliformula} yields
\begin{align}\label{toproductformula}
\hat{\lambda}_{E_t,v}(P_t)+\frac{1}{2}\left(\frac{m^2}{N}-m+\frac{N}{6}+2k\right)\log|u(t)|_v&=-\frac{1}{2}\left(\frac{m}{N^2}\log\left(\frac{|\al|_v^{2N}}{|\beta|_v^{m}}\right)-\log|\al|_v+\frac{\log|\beta|_v}{6} \right)\\
&-\log|\gamma|_v+o_v(1).
\end{align}
Finally, notice that \cite[Theorem \Rmnum{6}.4.2]{Silverman:Advanced} implies
$$\hat{\lambda}_{E,\ord_{t_0}}(P)=\ord_{t_0}(1-w)+\frac{1}{2}B_2\left(\frac{\ord_{t_0} w}{\ord_{t_0} \psi}\right)\ord_{t_0}\psi=\frac{1}{2}\left(\frac{m^2}{N}-m+\frac{N}{6}+2k\right).$$
Therefore
\begin{align*}
V_{P,t_0,v}(t_0)&=\displaystyle\lim_{t\stackrel{v}\to t_0}V_{P,t_0,v}(t)=-\frac{1}{2}\left(\frac{m}{N^2}\log\left(\frac{|\al|_v^{2N}}{|\beta|_v^{m}}\right)-\log|\al|_v+\frac{\log|\beta|_v}{6} \right)-\log|\gamma|_v\\
&=\log|x(t_0)|_v,
\end{align*}
where $x(t_0)=\frac{\beta^{m^2/2N^2-1/2}}{\gamma\al^{m/N-1/2}}$ belongs to a finite extension of $K$, denoted by $L$.

\end{proof}

\medskip
\proof[Proof of Proposition \ref{height is nonnegative}]
By \cite[Lemma \Rmnum{2}.2.2]{Silverman:VCHII} there is a finite map of smooth projective curves $B'\to B$ such that if $E'\to B'$ is a minimal model for $E\times_{B}B'$, then $E'$ has semi-stable reduction over the singular fibers of $E\to B$. Moreover, we may choose $B'$ so that everything is defined over $K$. Thus, by Lemma \ref{reductionlemma} and using the product formula, we may assume that the singular fibers of our elliptic surface $E \to B$ have multiplicative reduction.  

For all $t \in B(\Kbar)$ for which $E_t$ is smooth, we know from Proposition \ref{heightvsheight} that $h_P(t) = \hat{h}_{E_t}(P_t)$.  The canonical height is always non-negative, so we may conclude that $h_P(t)\geq 0$ for all such $t$.  

Assume now that $t_0 \in B(\Kbar)$ has a fiber with multiplicative reduction.  Enlarging the number field $K$ if necessary we may assume that $t_0\in B(K)$ and that its corresponding $x(t_0)$ defined in the statement of Lemma \ref{bad fibers} is in $K^{*}$. Then, on using the product formula, Lemma \ref{bad fibers} implies that $h_P(t_0) = 0$.  This completes the proof.  
\qed

\subsection{Proofs of Theorem \ref{good height} and Corollary \ref{equidistribution on B}}

\proof[Proof of Theorem \ref{good height}]
Let $\cL_P$ be the line bundle on $B$ induced from the divisor $D_E(P)$.  From Theorem \ref{metric properties}, we know that its $m$-th tensor power can be equipped with a continuous, adelic, semipositive metric, so that the corresponding height function is (a multiple of) the canonical height $\hat{h}_{E_t}(P_t)$ on the smooth fibers.  Thus, by pulling back the metric to $\cL_P$, we obtain a continuous, semipositive, adelic metric on $\cL_P$ inducing the desired height function.  

It remains to show that this height $h_P$ satisfies $h_P(B) = 0$.  This is a consequence of Propositions \ref{heightvsheight} and \ref{height is nonnegative} and Zhang's inequalities on successive minima \cite{Zhang:adelic}.  Recall that, since $\hat{h}_E(P) \not=0$, we know that there are infinitely many $t\in B(\Kbar)$ for which 
	$$\hat{h}_{E_t}(P_t) = 0.$$
(For a complex-dynamical proof, see \cite[Proposition 1.5, Proposition 2.3]{D:stableheight}.)  Therefore, from Proposition \ref{heightvsheight}, we may deduce that the essential minimum of $h_P$ on $B$ is equal to 0.  On the other hand, from Proposition \ref{height is nonnegative}, we know that $h_P(t) \geq 0$ for all $t \in B(\Kbar)$.  Therefore, from \cite[Theorem 1.10]{Zhang:adelic}, we may conclude that $h_P(B) = 0$. 
\qed

\proof[Proof of Corollary \ref{equidistribution on B}]
When combined with the equidistribution theorems of Yuan and Thullier \cite{Yuan:equidistribution, Thuillier:these}, we immediately obtain the corollary from Theorem \ref{good height}.  The measures $\mu_{P,v}$ are the curvature distributions associated to the metrics $\|\cdot\|_v$ at each place $v$.  From the definition of the metric in \S\ref{metric definition}, we see that they are given locally by 
	$$\mu_{P,v} = dd^c V_{P,t_0, v}(t)$$
in a $v$-adic neighborhood of any point $t_0\in B(\Kbar)$, and for any choice of uniformizer $u$ at $t_0$.  
\qed

\bigskip
\section{Proof of Theorem \ref{Zhang language}.}

\subsection{Reduction to the case of a fiber product of elliptic surfaces} \label{product reduction}

We first show that, to prove the theorem, it suffices to prove the result for sections of the fiber product $A = E_1 \times_B\cdots\times_BE_m$ of $m\ge 2$ elliptic surfaces $E_i\to B$ over the same base, and to assume that the line bundle $\cL$ is generated by the divisor $$\{O_{E_1}\}\times E_2\times\cdots E_m+E_1\times\{O_{E_2}\}\times\cdots\times E_m+\cdots + E_1\times E_2\times\cdots \times\{O_{E_m}\}.$$

Let $B$ be a quasiprojective smooth algebraic curve defined over $\Qbar$.  Suppose $A \to  B$ is family of abelian varieties defined over $\Qbar$ that is isogenous to a fibered product of $m\ge 2$ elliptic curves. That is, there is a branched cover $B' \to B$ and $m\ge2$ elliptic surfaces $E_i \to B'$  that give rise to an isogeny 
	$$E_1\times_{B'}\cdots\times_{B'} E_m \to A$$
over $B'$.  
Now let $\cL$ be a line bundle on $A$ which restricts to an ample and symmetric line bundle on each fiber $A_t$ for $t\in B$. Then the line bundle $\cL$ pulls back to a line bundle $\cL'$ on $E_1\times_{B'}\cdots\times_{B'} E_m$, and it again restricts to an ample and symmetric line bundle on each fiber over $t \in B'$. 

Now suppose that we have a section $P: B \to A$.  The section $P$ pulls back to a section $P': B' \to A$, and this in turn pulls back to a (possibly multi-valued) section of $E_1\times_{B'}\cdots\times_{B'} E_m$.  If multi-valued, we can perform a base change again, passing to a branched cover $B'' \to B'$, so that the induced section $P'': B'' \to E_1\times_{B''}\cdots\times_{B''} E_m$ is well defined.  By definition, the assumption that $P$ is non-special on $A$ means that it is non-special as a section of $E_1\times_{B''}\cdots\times_{B''} E_m$.  

Finally, we observe that the conclusion of Theorem \ref{Zhang language} does not depend on the choice of line bundle.  (We thank Joe Silverman for his help with this argument.)  Recall that, on any abelian variety $A$ defined over $\Qbar$, the notion of a ``small sequence" of points is independent of the choice of ample and symmetric line bundle.  That is, if we take two ample and symmetric divisors $D_1$ and $D_2$, then we know that there exists an integer $m_1>0$ so that $m_1 D_1 - D_2$ is ample; similarly there exists $m_2>0$ so that $m_2 D_2 - D_1$ is ample.  It follows from properties of the Weil height machine that the heights $h_{D_1}$ and $h_{D_2}$ will then satisfy
	$$\frac{1}{m_1} h_{D_2} + C_1 \leq h_{D_1} \leq m_2 \, h_{D_2} + C_2$$
for real constants $C_1, C_2$.   Upon passing to the canonical height, we conclude that 
\begin{equation} \label{comparable heights}
	\frac{1}{m_1} \hat{h}_{D_2} \leq \hat{h}_{D_1} \leq m_2 \, \hat{h}_{D_2}
\end{equation}
on the abelian variety.  In particular, $\hat{h}_{D_1}(a_i) \to 0$ for some sequence in $A(\Qbar)$ if and only if $\hat{h}_{D_2}(a_i) \to 0$.  Now suppose we have a family of abelian varieties $A \to B$.  Two line bundles $\cL_1$ and $\cL_2$ associated to relatively ample and symmetric divisors induce a canonical heights $\hat{h}_{\cL_{1,t}}$ and $\hat{h}_{\cL_{2,t}}$ on each fiber $A_t$.  But recalling that amplitude persists on Zariski open sets \cite[Theorem 1.2.17]{Lazarsfeld:Positivity:I}, there exist positive integers $m_1$ and $m_2$ so that the line bundles $\cL_1^{m_1}\otimes\cL_2^{-1}$ and $\cL_2^{m_2}\otimes \cL_1^{-1}$ are relatively ample on a Zariski open subset of the base $B$.  Passing to the canonical heights once again, we find that the relation \eqref{comparable heights} holds uniformly over $B$ (after possibly excluding finitely many points).  Therefore, for any section $P: B \to A$, there exists a positive constant $c(\cL_1, P)$ of Theorem \ref{Zhang language} for height $\hat{h}_{\cL_1}$ if and only if it there exists such a constant $c(\cL_2, P)$ for $\hat{h}_{\cL_2}$.  

\subsection{Proof for a fiber product of elliptic curves}
Fix integer $m\geq 2$, and let $E_i \to B$ for $i=1,\ldots,m$ be elliptic surfaces over the same base curve $B$, defined over $\Qbar$.  Let $A = E_1\times_B\cdots\times_B E_m$, and let $\cL$ be the line bundle on $E_1\times_B \cdots\times_B E_m$ associated to the divisor 
$$D = \{O_{E_1}\}\times E_2\times\cdots E_m+E_1\times\{O_{E_2}\}\times\cdots\times E_m+\cdots + E_1\times E_2\times\cdots \times\{O_{E_m}\}.$$
For all but finitely many $t \in B(\Qbar)$, the canonical height $\hat{h}_{\cL_t}$ on the fiber $A_t$ is easily seen to be the sum of canonical heights  (see, e.g., \cite{Hindry:Silverman} for properties of the height functions), so that
	$$\hat{h}_{\cL_t} = \displaystyle\sum_{i=1}^m\hat{h}_{E_{i,t}}.$$
	
Now assume that $P = (P_1, \ldots, P_m)$ is a section of $A\to B$.  Define
	$$\hat{h}_i(t) := \hat{h}_{E_{i,t}}(P_i(t))$$
for $i = 1,\ldots, m$ and for all $t \in B(\Qbar)$ where all $E_{i,t}$ are smooth elliptic curves.  Suppose there exists an infinite sequence $\{t_n\} \subset B(\Qbar)$ for which
\begin{equation}  \label{simultaneously small}
	\hat{h}_i(t_n) \to 0 \text{ for all }i=1,\ldots,m.
\end{equation}
as $n\to \infty$.  We will prove that for every pair $(i,j)$, there exists an infinite sequence $\{s_n\} \subset B(\Qbar)$ so that 
	$$\hat{h}_i(s_n) = \hat{h}_j(s_n) = 0$$
for all $n$.  In this way, we reduce our problem to the main results of \cite{Masser:Zannier:2, Masser:Zannier:nonsimple} which imply that the pair $(P_i, P_j)$ must be a special section of $E_i \times_B E_j$.  Finally, we observe that our definition of a special section $P = (P_1, P_2, \ldots, P_m)$ is equivalent to the statement that every pair $(P_i, P_j)$ is special.  Therefore, for any non-special section $P$, we can conclude that there exists a constant $c = c(P) >0$ so that the set
	$$\{t \in B(\Qbar): \hat{h}_{\cL_t}(P_t) < c\}$$ 
is finite.

Fix a pair $(i,j)$. First assume that neither $E_i$ nor $E_j$ is isotrivial.  If $P_i$ or $P_j$ is torsion, then the section $(P_i, P_j)$ is special.  Otherwise, we have $\hat{h}_{E_i}(P_i) \not=0$ and $\hat{h}_{E_j}(P_j) \not=0$, and we may apply Theorem \ref{good height} to deduce that the height functions $h_i$ and $h_j$ are ``good" on $B$.  More precisely, we let $M_i$ and $M_j$ be the adelically metrized line bundles on the base curve $B$ associated to the height functions $\hat{h}_i$ and $\hat{h}_j$, from Theorem \ref{good height}.  They are both equipped with continuous adelic metrics of non-negative curvature.  By assumption, we have 
\begin{equation}  \label{two are small}
	\hat{h}_i(t_n) \to 0 \quad \mbox{and} \quad \hat{h}_j(t_n) \to 0
\end{equation}
as $n\to \infty$.  Therefore, we may apply the observation of Chambert-Loir \cite[Proposition 3.4.2]{ChambertLoir:survey}, which builds upon on Zhang's inequalities \cite{Zhang:adelic}, to conclude that there exist integers $n_i$ and $n_j$ so that $M_i^{n_i}$ and $M_j^{n_j}$ are isomorphic as line bundles on $B$ and their metrics are scalar multiples of one another.  It follows that the height functions $\hat{h}_i$ and $\hat{h}_j$ are the same, up to scale, and in particular they have the same zero sets.  In other words, $P_i(t)$ is a torsion point on $E_{i,t}$ if and only if $P_j(t)$ is a torsion point on $E_{j,t}$ (for all but finitely many $t$ in $B$), and there are infinitely many such parameters $t\in B(\Qbar)$.

Now suppose that $E_i$ is isotrivial.  The existence of the small sequence $t_n$ in \eqref{two are small} implies that either $\hat{h}_{E_i}(P_i) \not=0$ or $P_i$ is torsion on $E_i$, and furthermore, if $P_i$ is torsion, then it follows that $(P_i, P_j)$ is a special section of $E_i\times_B E_j$.  Similarly if $E_j$ is isotrivial.  In other words, the existence of the sequence $t_n$ in \eqref{two are small} allows us to conclude that either $(P_i,P_j)$ is a special pair, or we have that both $\hat{h}_{E_i}(P_i) \not=0$ and $\hat{h}_{E_j}(P_j) \not=0$.  Therefore, we may proceed as above in the nonisotrivial case, applying Theorem \ref{good height} to deduce that the heights $\hat{h}_i$ and $\hat{h}_j$ coincide, up to scale, and in particular there are infinitely many parameters $s\in B(\Qbar)$ where 
	$$\hat{h}_i(s) = \hat{h}_j(s) = 0.$$
This concludes the proof of Theorem \ref{Zhang language}.

\bigskip
\section{Variation of canonical height, illustrated}
\label{experiment}

In this final section, we provide a few illustrations of the distributions $\mu_{P,v}$ for an archimedean place $v$, arising in Corollary \ref{equidistribution on B}.  In Proposition \ref{dense}, we present a complex-dynamical proof that the archimedean measures $\mu_{P,v}$ will have support equal to all of $B$.

\subsection{Images}  Given $E\to B$ and section $P$, we plot the parameters $t$ where $P_t$ is a torsion point on the fiber $E_t$ of specified order.  As proved in Corollary \ref{equidistribution on B}, the local height function at each place
	$$t \mapsto \hat{\lambda}_{E_t,v}(P_t)$$
determines the distribution of the torsion parameters; it is a potential for the measure $\mu_{P,v}$ (away from the singularities).  Recall that if we have two sections $P$ and $Q$ that are linearly related on $E$, then the distributions of their torsion parameters in $B$ will be the same.

Figure \ref{Example 1}, top, illustrates the example of Silverman from \cite[Theorem I.0.3]{Silverman:VCHI}.  Here, we have  
	$$E_t = \{y^2 + xy/t + y/t = x^3 + 2 x^2/t\}$$ 
with $B = \pr^1$ and $P_t = (0,0)$ in $(x,y)$-coordinates.  Plotted are the torsion parameters of orders $2^n$ for all $n\leq 8$; that is, the points $t$ in the base $B$ where $P_t$ is torsion of order $2^n$ on the fiber $E_t$.  Roughly, a smaller yellow dot corresponds to higher order of torsion.  Figure \ref{Example 1}, bottom, is another section of the same family, where the $x$-coordinate of $P_t$ is constant and equal to $-1/4$.  (Strictly speaking, this second $P$ is not a section of our given $E\to \pr^1$, because the $y$-coordinate will not lie in $\Kbar(B) \simeq \overline{\mathbb{Q}}(t)$ but in an extension; however, the property of being torsion and the determination of its order is independent of which point in the fiber we choose.)  Observe the distinctly different pattern of yellow dots in the first and second pictures, especially in the left half of the two pictures, illustrating the linear independence in $E(k)$ of the two sections.  

Figure \ref{Example 2} illustrates the torsion parameters for two independent sections of the Legendre family, 
	$$E_t = \{y^2 = x(x-1)(x-t)\}$$
over $B = \pr^1$, studied in \cite{Masser:Zannier}.   The chosen sections are $P_2$, with constant $x$-coordinate equal to $2$, and $P_5$, with constant $x$-coordinate equal to 5.  As in Figure \ref{Example 1}, we plot the torsion parameters of orders $2^n$ for all $n\leq 8$; generally, a smaller yellow dot signifies higher order of torsion.  It was proved in \cite{DWY:Lattes} that the limiting distributions for sections with constant $x$-coordinate satisfy $\mu_{P_x,\infty} = \mu_{P_{x'}, \infty}$ (at an archimedean place) if and only if $x=x'$.  It was proved in \cite{Stoll:torsion} and \cite{Mavraki:Weierstrass} that there are no $t \in \pr^1(\Kbar)$ for which both $(P_2)_t$ and $(P_5)_t$ are torsion on $E_t$.   Again, observe the difference in the geometry of the yellow dots for the two independent sections.   

Figure \ref{Example equidistribution} illustrates our equidistribution result, Corollary \ref{equidistribution on B}, for the example of the Legendre family with the section $P_5$. Plotted are the torsion parameters of orders $2^n$ with (a) $n\leq 6$, (b) $n\leq 8$, and (c) $n\leq 10$.  Observe how the yellow dots fill in the ``grid structure" in the base curve $B$, exactly as do the torsion points for one elliptic curve.  

\begin{figure} [h]
\includegraphics[width=4.5in]{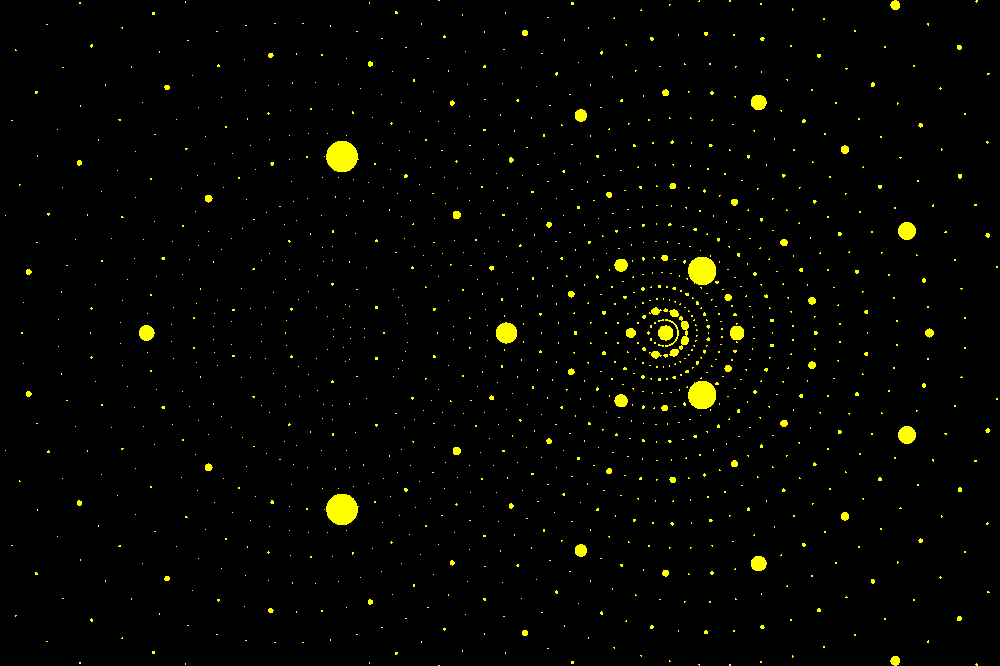}
\includegraphics[width=4.5in]{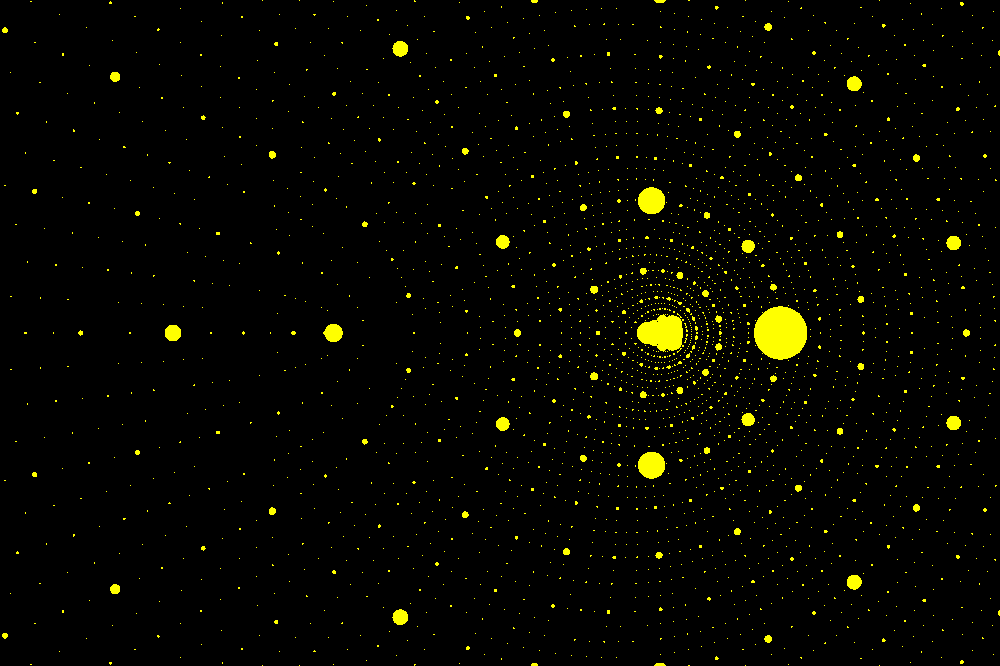}
\caption{ \small At top, Silverman's example from \cite[Theorem I.0.3]{Silverman:VCHI}, with $E_t = \{y^2 + xy/t + y/t = x^3 + 2 x^2/t\}$ and $P_t = (0,0)$, shown in the region $\{-2 \leq \operatorname{Re} t \leq 1, \; -1 \leq \operatorname{Im} t \leq 1\}$.  The singular fibers occur at $t = 0, -2/27, -1$, and one sees the effects of numerical error in a small neighborhood of these parameters.  At bottom, torsion parameters for section $P$ having $x$-coordinate $x(P_t) = -1/4$ for all $t$.}
\label{Example 1}
\end{figure}

\begin{figure} [h]
\includegraphics[width=3in]{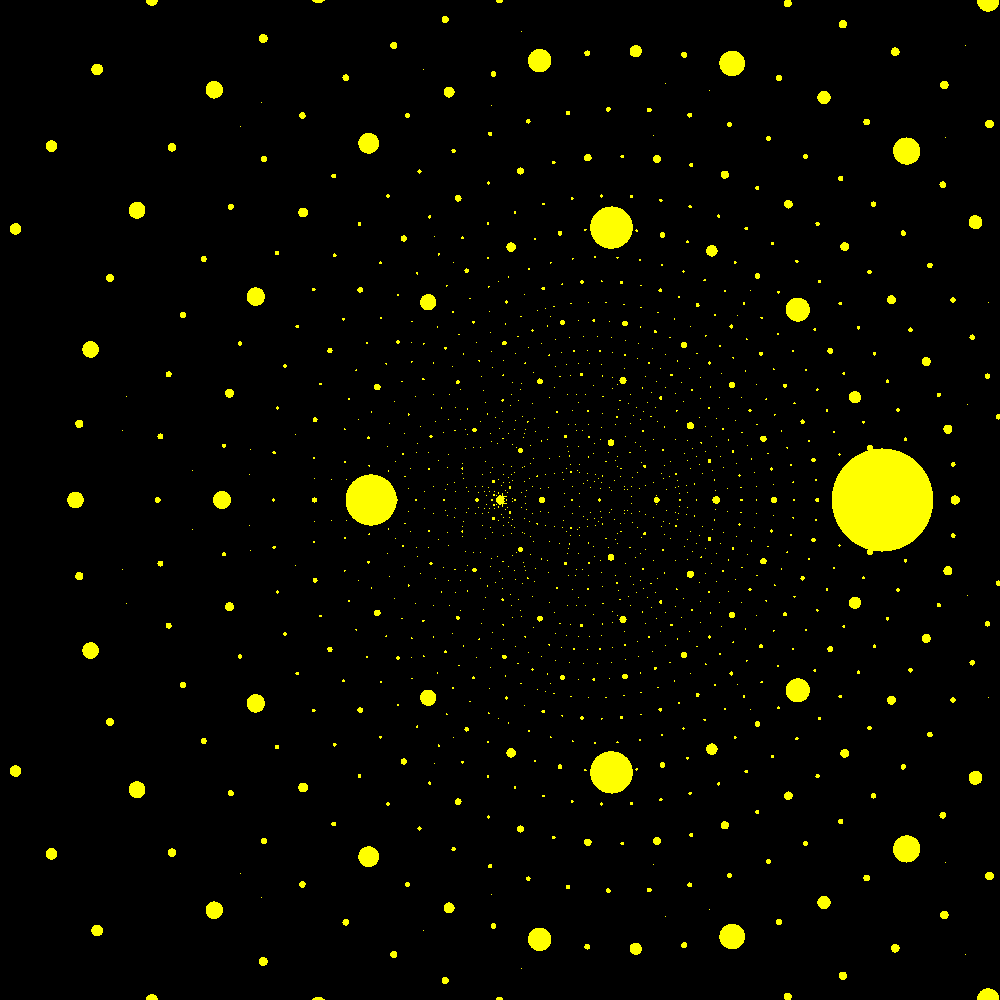} \quad
\includegraphics[width=3in]{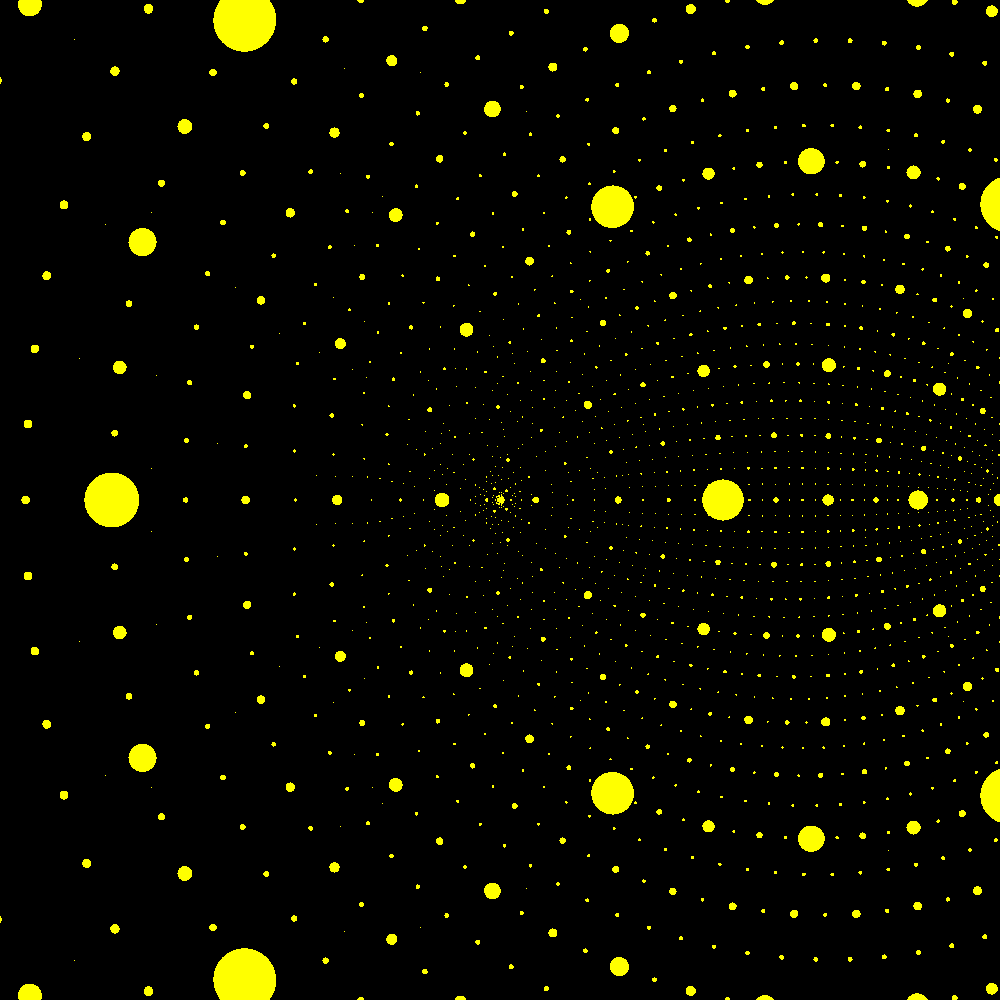} 
\caption{ \small  Torsion parameters for sections of the Legendre family, studied in \cite{Masser:Zannier}; here, $B = \pr^1$ and $E_t = \{y^2 = x(x-1)(x-t)\}$.  At left, the section $P_2$ with $x((P_2)_t) = 2$ for all $t$; at right, the section $P_5$ with $x((P_5)_t) = 5$ for all $t$.  Both are shown in the region $\{-3 \leq \operatorname{Re} t \leq 5, \; -4 \leq \operatorname{Im} t \leq 4\}$     }
\label{Example 2}
\end{figure}

\begin{figure} [h]
\includegraphics[width=3.0in]{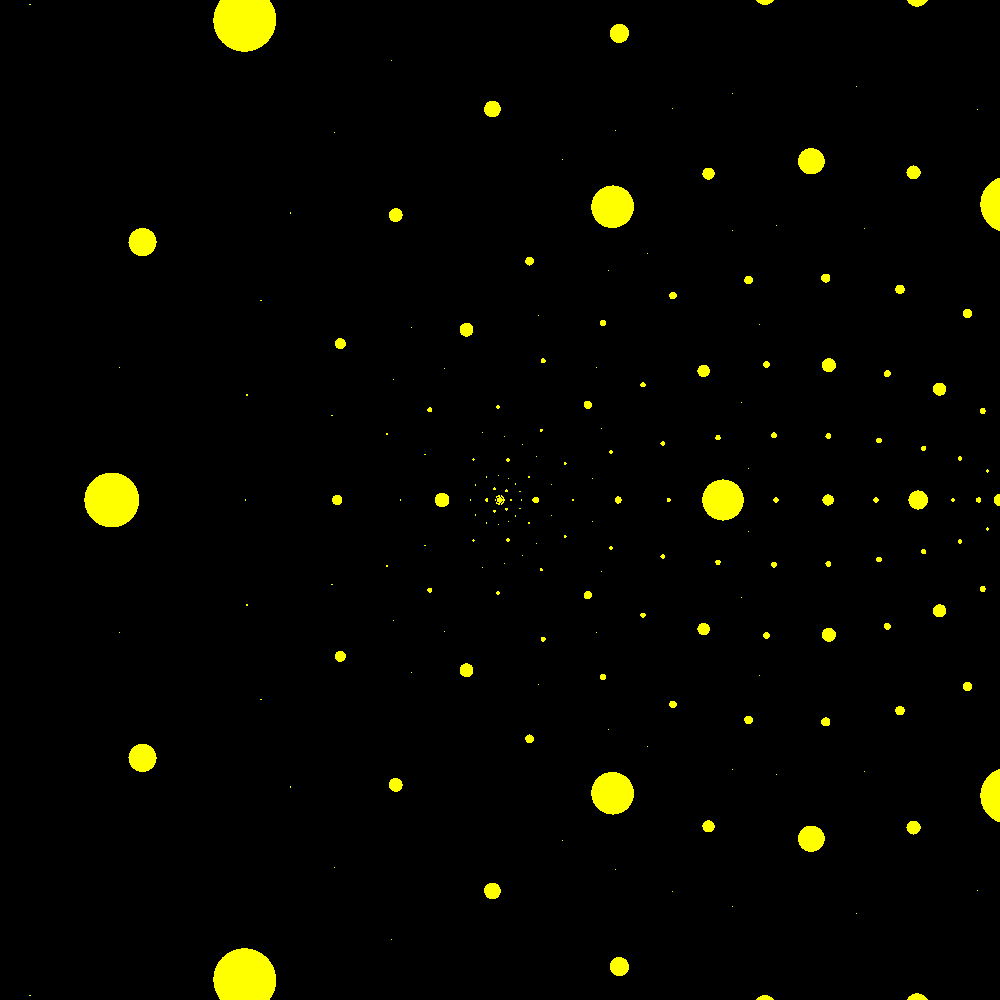} \\
\includegraphics[width=3.0in]{Lattes_torsion_5_8_10000.png} 
\includegraphics[width=3.0in]{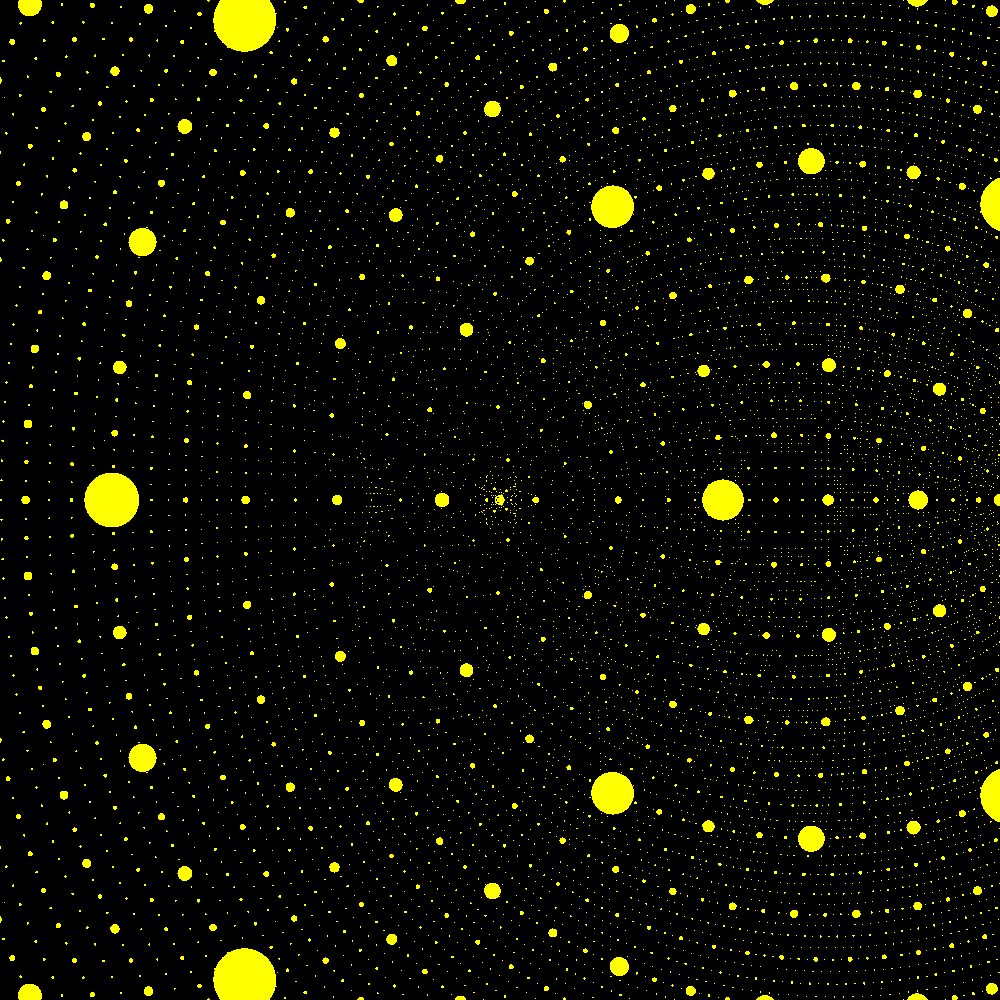} 
\caption{ \small  Illustrating equidistribution:  Torsion parameters of increasing orders for a section of the Legendre family, with $B = \pr^1$ and $E_t = \{y^2 = x(x-1)(x-t)\}$ and $P_5$ as in Figure \ref{Example 2}.  At top, torsion parameters of orders $2^n$ for $n\geq 6$; bottom left, of orders $2^n$ for $n\leq 8$, and bottom right, of orders $2^n$ for $n\leq 10$.    }
\label{Example equidistribution}
\end{figure}

\begin{remark}
As mentioned above, the smaller yellow dots in the illustrations correspond, roughly, to higher orders of torsion.  These images are produced with a standard escape-rate algorithm.  We use the dynamical system $f_t$ on $\pr^1$, induced from multiplication by 2 on the elliptic curve $E_t$ from Section \ref{dynamics}, line \eqref{projection}.  The coordinates on $\pr^1$ are chosen so that $\infty$ is the image of the 0 of $E_t$.  We mark $t$ yellow if $|f^n_t(\pi(P_t))| \geq 10000$ for some $n\leq 8$.  
\end{remark}

\subsection{Density of torsion parameters}
In all of these examples, the yellow dots will fill in the picture as the order of torsion grows, and the support of the measures $\mu_{P,v}$ is equal to $B(\bC)$.  In fact, this will always be the case, for any (nontrivial) section of a complex elliptic surface, as our final result, Proposition \ref{dense}, shows.  

Let $E\to B$ be an elliptic surface over a projective curve $B$, defined over $\bC$, and let $P: B\to E$ a section for which $\hat{h}_E(P) \not=0$ (over the function field $k = \bC(B)$).  Let $\mu_P$ be the measure on $B$ defined as in Proposition \ref{measure description}, as the pullback of the current $T$ that restricts to Haar measure on each smooth fiber.  In other words, $\mu_P$ is locally defined as the Laplacian of the function $G_P(t)$ introduced in Proposition \ref{escape rate subharmonic}, which is well defined when working over $\bC$.

\begin{proposition}  \label{dense}
Let $E\to B$ be an elliptic surface over a projective curve $B$, defined over $\bC$, and let $P: B\to E$ be a section for which $\hat{h}_E(P) \not=0$ (over the function field $k = \bC(B)$).  Then the set 
	$$\{t \in B:  P_t \mbox{ is torsion on } E_t\}$$
is dense in $B(\bC)$ and 
	$$\supp \mu_P = B(\bC).$$
\end{proposition}

We give a complex-dynamical proof, viewing Proposition \ref{dense} as a consequence of the main result of \cite{D:stableheight}.  (We do not use the equidistribution result, Corollary \ref{equidistribution on B}.)  An analytic proof is also presented in \cite[Notes to Chapter 3]{Zannier:book}.

\proof
Let $B^*\subset B$ be a finitely-punctured Riemann surface such that the fiber $E_t$ is smooth for all $t \in B^*$.  Let $\pi_t: E_t\to \pr^1$ be the degree-two projection and $f_t: \pr^1 \to \pr^1$ be the rational map induced by multiplication-by-2 on $E_t$, as defined in the introduction to Section \ref{dynamics}.  It is well known that the holomorphic family $\{f_t:  t\in B^*\}$ is structurally stable; see, e.g., \cite[Chapter 4]{McMullen:CDR}.  Thus, over any simply-connected subset $U$ of $B^*$, there is a holomorphic motion of the periodic points of $f_t$ which extends uniquely to a holomorphic motion of all of $\pr^1$, conjugating the dynamics.  

The key observation is that $\mu_P$ is precisely the ``bifurcation measure" of the pair $(f,P)$ on $B^*$.  See \cite[\S2.7]{D:KAWA} and \cite{D:stableheight} for definitions.  The support of $\mu_P$ is equal to the bifurcation locus of $(f,P)$; in particular, the parameters $t\in B^*$ for which $\pi_t(P_t)$ is preperiodic for $f_t$ are dense in $\supp \mu_P$.  Therefore, it suffices to show that $\supp \mu_P = B$.

Suppose to the contrary that there is an open disk $U \subset B^*$ for which $\mu_P(U) = 0$.  Then the pair $(f,P)$ is stable on $U$, and therefore $\pi_t(P_t)$ cannot be a repelling periodic point for any $t \in U$.  From the uniqueness of the holomorphic motion, it follows that $t\mapsto \pi_t(P_t)$ is part of the holomorphic motion on $U$.  By analytic continuity, then, we deduce that $\pi_t(P_t)$ must follow the motion of a point over all of $B^*$.  This implies that the pair $(f,P)$ is stable throughout $B^*$ and the measure $\mu_P$ is 0.  But this is absurd by the assumption that $\hat{h}_E(P) \not=0$; see \cite[Theorem 1.1]{D:stableheight}.  
\qed

\bigskip \bigskip
\def\cprime{$'$}

\end{document}